\documentclass[10pt]{article}
\usepackage{amsmath,amscd}
\usepackage{amssymb,latexsym,amsthm}
\usepackage{color}
\usepackage[spanish,english]{babel}
\usepackage{amssymb}
\usepackage{color}
\usepackage{amsmath,amsthm,amscd}
\usepackage[latin1]{inputenc}
\usepackage{lscape}
\usepackage{fancyhdr}
\usepackage{amsfonts}
\usepackage{pb-diagram}

\numberwithin{equation}{section}
\newtheorem{theorem}{Theorem}[section]

\newtheorem{definition}[theorem]{Definition}
\newtheorem{proposition}[theorem]{Proposition}
\newtheorem{lemma}[theorem]{Lemma}

\newtheorem{conjecture}[theorem]{Conjecture}
\newtheorem{corollary}[theorem]{Corollary}

\theoremstyle{definition}

\newtheorem{example}[theorem]{Example}

\newtheorem{remark}[theorem]{Remark}

\newcommand{\cU}{\mbox{${\cal U}$}}

\newcommand{\cW}{\mbox{${\cal W}$}}

\title{\textbf{Some open problems in the context of\\ skew $PBW$ extensions and semi-graded rings}}
\author{Oswaldo Lezama\\
\texttt{jolezamas@unal.edu.co}
\\ Seminario de Álgebra Constructiva - SAC$^2$\\ Departamento de Matemáticas\\ Universidad Nacional de
Colombia, Sede Bogot\'a}
\date{}
\begin{document}
\maketitle
\begin{abstract}
\noindent In this paper we discuss some open problems of
non-commutative algebra and non-commutative algebraic geometry
from the approach of skew $PBW$ extensions and semi-graded rings.
More exactly, we will analyze the isomorphism arising in the
investigation of the Gelfand-Kirillov conjecture about the
commutation between the center and the total ring of fractions of
an Ore domain. The Serre's conjecture will be discussed for a
particular class of skew $PBW$ extensions. The questions about the
noetherianity and the Zariski cancellation property of
Artin-Schelter regular algebras will be reformulated for
semi-graded rings. Advances for the solution of some of the
problems are included.
\bigskip

\noindent \textit{Key words and phrases.} Gelfand-Kirillov
conjecture, Serre's conjecture, Artin-Schelter regular algebras,
Zariski cancellation problem, skew $PBW$ extensions, semi-graded
rings.

\bigskip

\noindent 2010 \textit{Mathematics Subject Classification.}
Primary: 16S36. Secondary: 16U20, 16D40, 16E05, 16E65, 16S38,
16S80, 16W70, 16Z05.
\end{abstract}

\section{Introduction}

Commutative algebra and algebraic geometry are source of many
interesting problems of non-commutative algebra and geometry. For
example, the famous theorem of Serre of commutative projective
algebraic geometry states that the category of coherent sheaves
over the projective $n$-space $\mathbb{P}^n$ is equivalent to a
category of noetherian graded modules over a graded commutative
polynomial ring. The study of this equivalence for non-commutative
algebras gave origin to an intensive study of the so-called
non-commutative projective schemes associated to non-commutative
finitely graded noetherian algebras, and produced a beautiful
result, due to Artin and Zhang, and also proved independently by
Verevkin, known as the non-commutative version of Serre's theorem.
Another famous example that has to be mentioned is the Zariski
cancellation problem on the algebra of commutative polynomials
$K[x_1,\dots,x_n]$ over a field $K$, this problem asks if given a
commutative $K$-algebra $B$ and an isomorphism
$K[x_1,\dots,x_n][t]\cong B[t]$, follows that
$K[x_1,\dots,x_n]\cong B$. This problem has been reformulated for
non-commutative algebras and occupied the attention of many
researchers in the last years. In this paper we will present some
famous open problems of non-commutative algebra and
non-commutative algebraic geometry, but interpreting them from the
approach of skew $PBW$ extensions and semi-graded rings. For some
of these problems we will include some advances. Skew $PBW$
extension were define first in \cite{LezamaGallego}, many
important algebras coming from mathematical physics are particular
examples of skew $PBW$ extensions: The enveloping algebra
$\cal{U}(\cal{G})$, where $\cal{G}$ is a finite dimensional Lie
algebra, the algebra of $q$-differential operators, the algebra of
shift operators, the additive analogue of the Weyl algebra, the
multiplicative analogue of the Weyl algebra, the quantum algebra
$\mathcal{U}'(so(3,K))$, $3$-dimensional skew polynomial algebras,
the dispin algebra, the Woronowicz algebra, the $q$-Heisenberg
algebra, among many others (see \cite{lezamareyes1} for the
definition of these algebras). On the other hand, the semi-graded
rings generalize the finitely graded algebras and the skew PBW
extensions, this general class of non-commutative rings was
introduced in \cite{lezamalatorre}.

The problems that we will analyze and formulate from the general
approach of skew $PBW$ extensions and semi-graded rings are: 1)
Investigate the Serre-Artin-Zhang-Verevkin theorem (see
\cite{Artin2}) for semi-graded rings. 2) Study the
Gelfand-Kirillov conjecture (\cite{GK}) for bijective skew $PBW$
extensions over Ore domains. Moreover, investigate for them the
commutation between the center and the total ring of fractions. 3)
Give a matrix-constructive proof of the Quillen-Suslin theorem
(see \cite{Serre}, \cite{Quillen} and \cite{Suslin1}) for a
particular class of skew polynomials rings. Include an algorithm
that computes a basis of a given finitely generated projective
module. 4) Investigate if the semi-graded Artin-Schelter regular
algebras (\cite{Artin1}) are noetherian domains. 5) Investigate
the cancellation property (see \cite{BellZhang}) for a particular
class of semi-graded Artin-Schelter regular algebras of global
dimension $3$.

Despite of the paper is a revision of previous works, written with
the purpose of reformulating some famous open problems in a more
general context, some new results are included. The first problem
was solved in \cite{lezamalatorre} assuming that the semi-graded
rings are domains, in the present paper we will prove an improved
result avoiding the restriction about the absence zero divisors.
With respect to the problem 4), we give an advance showing that a
particular class of semi-graded Artin-Schelter regular algebras
are noetherian domains. Moreover, all of examples of semi-graded
Artin-Schelter regular algebras exhibited in the present paper are
noetherian domains. The new results of the paper are concentrated
in Theorem \ref{theorem16.7.13}, Corollary \ref{corollary17.5.14}
and Theorem \ref{theorem18.4.12}.

The precise description of the classical famous open problems, and
the new reformulation of them in the context of the present paper,
will be given below in each section. We have included basic
definitions and facts about finitely graded algebras, skew $PBW$
extensions and semi-graded rings, needed for a better
understanding of the paper. If not otherwise noted, all modules
are left modules; $B$ will denote a non-commutative ring; $Z(B)$
is the center of $B$; $M_s(B)$ be the ring of $s\times s$ matrices
over $B$, and $ GL_s(B):=\{F\in M_s(B)\mid F\, \text{is invertible
in}\, M_s(B)\}$ be the general linear group over $S$; $K$ will be
a field.

\subsection{Finitely graded algebras}\label{subsection1.1}

Our first problem concerns with the generalization of a famous
theorem of non-commutative projective algebraic geometry over
finitely graded algebras. In this subsection we review the theorem
and some well-known basic facts of finitely graded algebras (see
\cite{Rogalski}).

\begin{definition}
Let $K$ be a field. It is said that a $K$-algebra $A$ is finitely
graded if the following conditions hold:
\begin{itemize}
\item $A$ is $\mathbb{N}$-graded: $A=\bigoplus_{n\geq 0} A_n$.
\item $A$ is connected, i.e., $A_0=K$.
\item $A$ is finitely generated as $K$-algebra.
\end{itemize}
\end{definition}
The most remarkable examples of finitely graded algebras for which
non-commutative projective algebraic geometry has been developed
are the quantum plane, the Jordan plane, the Sklyanin algebra and
the multi-parametric quantum affine $n$-space (see
\cite{Rogalski}).

Let $A$ be a finitely graded $K$-algebra and $M=\bigoplus_{n\in
\mathbb{Z}} M_n$ be a $\mathbb{Z}$-graded $A$-module which is
finitely generated. Then,
\begin{enumerate}
\item[\rm (i)]For every $n\in \mathbb{Z}$, $\dim_{K}M_n<\infty$.
\item[\rm (ii)]The \textit{Hilbert series} of $A$ is defined by
\begin{center}
$h_A(t):=\sum_{n=0}^\infty(\dim_K A_n)t^n$.
\end{center}
\item[\rm (iii)]The \textit{Gelfand-Kirillov dimension} of $A$ is defined by
\begin{equation}\label{equ17.2.2}
{\rm GKdim}(A):=\sup_{V}\overline{\lim_{n\to \infty}}\log_n\dim_K
V^n,
\end{equation}
where $V$ ranges over all frames of $A$ and $V^n:=\, _K\langle
v_1\cdots v_n| v_i\in V\rangle$ (a \textit{frame} of $A$ is a
finite dimensional $K$-subspace of $A$ such that $1\in V$; since
$A$ is a $K$-algebra, then $K\hookrightarrow A$, and hence, $K$ is
a frame of $A$ of dimension $1$).
\item[\rm (iv)]A famous theorem of Serre on commutative projective algebraic geometry states that
the category of coherent sheaves over the projective $n$-space
$\mathbb{P}^n$ is equivalent to a category of noetherian graded
modules over a graded commutative polynomial ring. The study of
this equivalence for non-commutative finitely graded noetherian
algebras is known as the non-commutative version of Serre's
theorem and is due to Artin, Zhang and Verevkin. In the next
numerals we give the ingredients needed for the formulation of
this theorem (see \cite{Artin2}, \cite{Verevkin}).
\item[\rm (v)]Suppose that $A$ is left noetherian. Let ${\rm gr}-A$ be the abelian category of
finitely generated $\mathbb{Z}$-graded left $A$-modules. The
abelian category ${\rm qgr}-A$ is defined in the following way:
The objects are the same as the objects in ${\rm gr}-A$, and we
let $\pi:{\rm gr}-A\to {\rm qgr}-A$ be the identity map on the
objects. The morphisms in ${\rm qgr}-A$ are defined in the
following way:
\begin{center}
$Hom_{\rm {qgr}-A}(\pi(M),\pi(N)):=\underrightarrow{\lim}Hom_{{\rm
gr}-A}(M_{\geq n}, N/T(N))$,
\end{center}
where the direct limit is taken over maps of abelian groups
\begin{center}
$Hom_{{\rm gr}-A}(M_{\geq n}, N/T(N))\to Hom_{{\rm gr}-A}(M_{\geq
n+1}, N/T(N))$
\end{center}
induced by the inclusion homomorphism $M_{\geq n+1}\to M_{\geq
n}$; $T(N)$ is the torsion submodule of $N$ and an element $x\in
N$ is torsion if $A_{\geq n}x=0$ for some $n\geq 0$. The pair
$(\rm {qgr}-A,\pi(A))$ is called the \textit{non-commutative
projective scheme} associated to $A$, and denoted by $\rm
{qgr}-A$. Thus, ${\rm qgr}-A$ is a quotient category, ${\rm
qgr}-A={\rm gr}-A/{\rm tor}-A$.
\item[\rm (vi)]\textit{$($Serre's theorem$)$ Let $A$ be a commutative finitely graded $K$-algebra generated in degree $1$.
Then, there exists an equivalence of categories}
\begin{center}
${\rm qgr}-A\simeq {\rm coh}({\rm proj}(A))$.
\end{center}
\textit{In particular},
\begin{center}
${\rm qgr}-K[x_0,\dots,x_n]\simeq {\rm coh}(\mathbb{P}^n)$.
\end{center}
\item[\rm (vii)]Suppose that $A$ is left noetherian and let $i\geq 0$, it is said that $A$ satisfies the $\chi_i$ condition if for every
finitely generated $\mathbb{Z}$-graded $A$-module $M$,
$\dim_K(\underline{Ext}_A^j(K,M))<\infty$ for any $j\leq i$; the
algebra $A$ satisfies the $\chi$ condition if it satisfies the
$\chi_i$ condition for all $i\geq 0$ (since $K$ is finitely
generated as $A$-module, then
$\dim_K(\underline{Ext}_A^j(K,M))=\dim_K(Ext_A^j(K,M))$).
\item[\rm (viii)]\textit{$($Serre-Artin-Zhang-Verevkin theorem$)$ If $A$ is left noetherian and satisfies $\chi_1$,
then $\Gamma(\pi(A))_{\geq 0}$ is left noetherian and there exists
an equivalence of categories
\begin{equation}\label{equation1.2}
{\rm qgr}-A\simeq {\rm qgr}-\Gamma(\pi(A))_{\geq 0},
\end{equation}\label{equivalenceSAZV}
where $\Gamma(\pi(A))_{\geq 0}:=\bigoplus_{d=0}^\infty Hom_{{\rm
qgr}-A}(\pi(A),s^d(\pi(A)))$ and $s$ is the autoequivalence of
${\rm qg}r-A$ defined by the shifts of degrees.}
\end{enumerate}

\subsection{Skew $PBW$ extensions}

The second ingredient needed for setting some of the problems that
we will discuss in the present paper are the skew $PBW$ extensions
introduced first in \cite{LezamaGallego}. Skew $PBW$ extensions
are non-commutative rings of polynomial type and cover many
examples of quantum algebras and rings coming from mathematical
physics: Habitual ring of polynomials in several variables, Weyl
algebras, enveloping algebras of finite dimensional Lie algebras,
algebra of $q$-differential operators, many important types of Ore
algebras, algebras of diffusion type, additive and multiplicative
analogues of the Weyl algebra, dispin algebra
$\mathcal{U}(osp(1,2))$, quantum algebra $\mathcal{U}'(so(3,K))$,
Woronowicz algebra $\mathcal{W}_{\nu}(\mathfrak{sl}(2,K))$, Manin
algebra $\mathcal{O}_q(M_2(K))$, coordinate algebra of the quantum
group $SL_q(2)$, $q$-Heisenberg algebra \textbf{H}$_n(q)$, Hayashi
algebra $W_q(J)$, differential operators on a quantum space
$D_{\textbf{q}}(S_{\textbf{q}})$, Witten's deformation of
$\mathcal{U}(\mathfrak{sl}(2,K))$, multiparameter Weyl algebra
$A_n^{Q,\Gamma}(K)$, quantum symplectic space
$\mathcal{O}_q(\mathfrak{sp}(K^{2n}))$, some quadratic algebras in
3 variables, some 3-dimensional skew polynomial algebras,
particular types of Sklyanin algebras, among many others. For a
precise definition of any of these rings and algebras see
\cite{lezamareyes1}

\begin{definition}[\cite{LezamaGallego}]\label{gpbwextension}
Let $R$ and $A$ be rings. We say that $A$ is a \textit{skew $PBW$
extension of $R$} $($also called a $\sigma-PBW$ extension of
$R$$)$ if the following conditions hold:
\begin{enumerate}
\item[\rm (i)]$R\subseteq A$.
\item[\rm (ii)]There exist finitely many elements $x_1,\dots ,x_n\in A$ such $A$ is a left $R$-free module with basis
\begin{center}
${\rm Mon}(A):= \{x^{\alpha}=x_1^{\alpha_1}\cdots
x_n^{\alpha_n}\mid \alpha=(\alpha_1,\dots ,\alpha_n)\in
\mathbb{N}^n\}$, with $\mathbb{N}:=\{0,1,2,\dots\}$.
\end{center}
The set ${\rm Mon}(A)$ is called the set of standard monomials of
$A$.
\item[\rm (iii)]For every $1\leq i\leq n$ and $r\in R-\{0\}$ there exists $c_{i,r}\in R-\{0\}$ such that
\begin{equation}\label{sigmadefinicion1}
x_ir-c_{i,r}x_i\in R.
\end{equation}
\item[\rm (iv)]For every $1\leq i,j\leq n$ there exists $c_{i,j}\in R-\{0\}$ such that
\begin{equation}\label{sigmadefinicion2}
x_jx_i-c_{i,j}x_ix_j\in R+Rx_1+\cdots +Rx_n.
\end{equation}
Under these conditions we will write $A:=\sigma(R)\langle
x_1,\dots ,x_n\rangle$.
\end{enumerate}
\end{definition}
Associated to a skew $PBW$ extension $A=\sigma(R)\langle x_1,\dots
,x_n\rangle$ there are $n$ injective endomorphisms
$\sigma_1,\dots,\sigma_n$ of $R$ and $\sigma_i$-derivations, as
the following proposition shows.

\begin{proposition}[\cite{LezamaGallego}, Proposition 3]\label{sigmadefinition}
Let $A$ be a skew $PBW$ extension of $R$. Then, for every $1\leq
i\leq n$, there exist an injective ring endomorphism
$\sigma_i:R\rightarrow R$ and a $\sigma_i$-derivation
$\delta_i:R\rightarrow R$ such that
\begin{center}
$x_ir=\sigma_i(r)x_i+\delta_i(r)$,
\end{center}
for each $r\in R$.
\end{proposition}
From Definition \ref{gpbwextension} (iv), there exists a unique
finite set of constants $c_{ij}, d_{ij}, a_{ij}^{(k)}\in R$ such
that
\begin{equation}\label{equation1.2.1}
x_jx_i=c_{ij}x_ix_j+a_{ij}^{(1)}x_1+\cdots+a_{ij}^{(n)}x_n+d_{ij},
\ \text{for every}\  1\leq i,j\leq n.
\end{equation}

A particular case of skew $PBW$ extension is when all derivations
$\delta_i$ are zero. Another interesting case is when all
$\sigma_i$ are bijective and the constants $c_{ij}$ are
invertible.

\bigskip

\begin{definition}[\cite{LezamaGallego}]\label{sigmapbwderivationtype}
Let $A$ be a skew $PBW$ extension.
\begin{enumerate}
\item[\rm (a)]
$A$ is quasi-commutative if the conditions {\rm(}iii{\rm)} and
{\rm(}iv{\rm)} in Definition \ref{gpbwextension} are replaced by
\begin{enumerate}
\item[\rm (iii')]For every $1\leq i\leq n$ and $r\in R-\{0\}$ there exists $c_{i,r}\in R-\{0\}$ such that
\begin{equation}
x_ir=c_{i,r}x_i.
\end{equation}
\item[\rm (iv')]For every $1\leq i,j\leq n$ there exists $c_{i,j}\in R-\{0\}$ such that
\begin{equation}
x_jx_i=c_{i,j}x_ix_j.
\end{equation}
\end{enumerate}
\item[\rm (b)]$A$ is bijective if $\sigma_i$ is bijective for
every $1\leq i\leq n$ and $c_{i,j}$ is invertible for any $1\leq
i<j\leq n$.
\end{enumerate}
\end{definition}
Observe that quasi-commutative skew $PBW$ extensions are
$\mathbb{N}$-graded rings, but arbitrary skew $PBW$ extensions are
semi-graded rings as we will see in the next subsection. Actually,
the main motivation for constructing the non-commutative algebraic
geometry of semi-graded rings is due to arbitrary skew $PBW$
extensions.

Many properties of skew $PBW$ extensions have been studied in
previous works (see \cite{Lezama-OreGoldie}, \cite{Acosta2},
\cite{lezamareyes1}, \cite{reyesA}, \cite{reyessuarez},
\cite{reyessuarez2}, \cite{ReyesSuarez}). The next theorem
establishes two ring theoretic results for skew $PBW$ extensions.

\begin{theorem}[\cite{lezamareyes1}, \cite{Lezama-OreGoldie}]\label{1.3.4}
Let $A$ be a bijective skew $PBW$ extension of a ring $R$.
\begin{enumerate}
\item[\rm (i)]$($Hilbert Basis Theorem$)$ If $R$ is a left $($right$)$ Noetherian ring then $A$ is also left $($right$)$ Noetherian.
\item[\rm (ii)]$($Ore's theorem$)$ If $R$ is a left Ore
domain $R$. Then $A$ is also a left Ore domain.
\end{enumerate}
\end{theorem}

\subsection{Semi-graded rings}

\noindent The third notion needed for setting the problems
analyzed in the paper are the semi-graded rings that we will
recall in this subsection. In \cite{lezamalatorre} were proved
some properties of them, in particular, it was showed that graded
rings, finitely graded algebras and skew $PBW$ extensions are
particular cases of this type of non-commutative rings.

\begin{definition}
Let $B$ be a ring. We say that $B$ is semi-graded $(SG)$ if there
exists a collection $\{B_n\}_{n\geq 0}$ of subgroups $B_n$ of the
additive group $B^+$ such that the following conditions hold:
\begin{enumerate}
\item[\rm (i)]$B=\bigoplus_{n\geq 0}B_n$.
\item[\rm (ii)]For every $m,n\geq 0$, $B_mB_n\subseteq B_0\oplus \cdots \oplus B_{m+n}$.
\item[\rm (iii)]$1\in B_0$.
\end{enumerate}
The collection $\{B_n\}_{n\geq 0}$ is called a semi-graduation of
$B$ and we say that the elements of $B_n$ are homogeneous of
degree $n$. Let $B$ and $C$ be semi-graded rings and let $f: B\to
C$ be a ring homomorphism, we say that $f$ is homogeneous if
$f(B_n)\subseteq C_{n}$ for every $n\geq 0$.
\end{definition}

\begin{definition}
Let $B$ be a $SG$ ring and let $M$ be a $B$-module. We say that
$M$ is a $\mathbb{Z}$-semi-graded, or simply semi-graded, if there
exists a collection $\{M_n\}_{n\in \mathbb{Z}}$ of subgroups $M_n$
of the additive group $M^+$ such that the following conditions
hold:
\begin{enumerate}
\item[\rm (i)]$M=\bigoplus_{n\in \mathbb{Z}} M_n$.
\item[\rm (ii)]For every $m\geq 0$ and $n\in \mathbb{Z}$, $B_mM_n\subseteq \bigoplus_{k\leq m+n}M_k$.
\end{enumerate}
We say that $M$ is positively semi-graded, also called
$\mathbb{N}$-semi-graded, if $M_n=0$ for every $n<0$. Let $f: M\to
N$ be an homomorphism of $B$-modules, where $M$ and $N$ are
semi-graded $B$-modules, we say that $f$ is homogeneous if
$f(M_n)\subseteq N_n$ for every $n\in \mathbb{Z}$.
\end{definition}

\begin{definition}
Let $B$ be a $SG$ ring and $M$ be a semi-graded module over $B$.
Let $N$ be a submodule of $M$, we say that $N$ is a semi-graded
submodule of $M$ if $N=\bigoplus_{n\in \mathbb{Z}}N_n$.
\end{definition}

\begin{definition}\label{definition17.5.4}
Let $B$ be a ring. We say that $B$ is finitely semi-graded $(FSG)$
if $B$ satisfies the following conditions:
\begin{enumerate}
\item[\rm (i)]$B$ is $SG$.
\item[\rm (ii)]There exists finitely many elements $x_1,\dots,x_n\in B$ such that the
subring generated by $B_0$ and $x_1,\dots,x_n$ coincides with $B$.
\item[\rm (iii)]For every $n\geq 0$, $B_n$ is a free $B_0$-module of finite dimension.
\end{enumerate}
Moreover, if $M$ is a $B$-module, we say that $M$ is finitely
semi-graded if $M$ is semi-graded, finitely generated, and for
every $n\in \mathbb{Z}$, $M_n$ is a free $B_0$-module of finite
dimension.
\end{definition}

\begin{remark}
Observe if $B$ is $FSG$, then $B_0B_p=B_p$ for every $p\geq 0$,
and if $M$ is finitely semi-graded, then $B_0M_n=M_n$ for all
$n\in \mathbb{Z}$.
\end{remark}

From the definitions above we get the following conclusions.

\begin{proposition}[\cite{lezamalatorre}]\label{proposition17.5.5}
Let $B=\bigoplus_{n\geq 0}B_n$ be a $SG$ ring and $I$ be a proper
two-sided ideal of $B$ semi-graded as left ideal. Then,
\begin{enumerate}
\item[\rm (i)]$B_0$ is a subring of $B$. Moreover, for any $n\geq 0$, $B_0\oplus \cdots \oplus B_{n}$ is a $B_0-B_0$-bimodule, as well as $B$.
\item[\rm (ii)]$B$ has a standard $\mathbb{N}$-filtration given by
\begin{equation}\label{equ17.5.1}
F_n(B):=B_0\oplus \cdots \oplus B_{n}.
\end{equation}
\item[\rm (iii)]The associated graded ring $Gr(B)$ satisfies
\begin{center}
$Gr(B)_n\cong B_n$, for every $n\geq 0$ $($isomorphism of abelian
groups$)$.
\end{center}
\item[\rm (iv)]Let $M=\bigoplus_{n\in \mathbb{Z}}M_n$ be a semi-graded $B$-module and $N$ a submodule of $M$. The following conditions are equivalent:
\begin{enumerate}
\item[\rm (a)]$N$ is semi-graded.
\item[\rm (b)]For every $z\in N$, the homogeneous components of $z$ are in $N$.
\item[\rm (c)]$M/N$ is semi-graded with semi-graduation given by
\begin{center}
$(M/N)_n:=(M_n+N)/N$, $n\in \mathbb{Z}$.
\end{center}
\end{enumerate}
\item[\rm (v)]$B/I$ is $SG$.
\end{enumerate}
\end{proposition}

\begin{proposition}[\cite{lezamalatorre}]\label{proposition16.5.7}
{\rm (i)} Any $\mathbb{N}$-graded ring is $SG$.

{\rm (ii)} Let $K$ be a field. Any finitely graded $K$-algebra is
a $FSG$ ring.

{\rm (iii)} Any skew $PBW$ extension is a $FSG$ ring.
\end{proposition}

\begin{remark}
(i) In \cite{lezamalatorre} was proved that the previous
inclusions are proper.

(ii) The class of $FSG$ rings also includes properly the multiple
Ore extensions introduced in \cite{XuHuangWang}.
\end{remark}

\subsection{Problem 1}

\begin{itemize}
\item \textit{Investigate the Serre-Artin-Zhang-Verevkin theorem for semi-graded rings}.
\end{itemize}

This problem was partially solved in \cite{lezamalatorre} where it
was assumed that the semi-graded left noetherian ring $B$ is a
domain, however, the Serre-Artin-Zhang-Verevkin theorem for
finitely graded algebras does not include this restriction. Next
we will present the main ingredients of an improved proof where
the domain restriction has been removed.

Let $B=\bigoplus_{n\geq 0}B_n$ be a $SG$ ring that satisfies the
following conditions:

(C1) $B$ is left noetherian.

(C2) $B_0$ is left noetherian.

(C3) For every $n$, $B_n$ is a finitely generated left
$B_0$-module.

(C4) $B_0\subset Z(B)$.

\begin{proposition}[\cite{lezamalatorre}, Proposition 5.2]
Let ${\rm sgr}-B$ be the collection of all finitely generated
semi-graded $B$-modules, then ${\rm sgr}-B$ is an abelian category
where the morphisms are the homogeneous $B$-homomorphisms.
\end{proposition}

Definition 5.3 in \cite{lezamalatorre} can be improved in the
following way.

\begin{definition}\label{definition16.5.30}
Let $M$ be an object of ${\rm sgr}-B$.
\begin{enumerate}
\item[\rm (i)]For  $s\geq 0$, $B_{\geq s}$ is the
least two-sided ideal of $B$ that satisfies the following
conditions:
\begin{enumerate}
\item[\rm (a)]$B_{\geq s}$ contains $\bigoplus_{p\geq s}B_p$.
\item[\rm (b)]$B_{\geq s}$ is semi-graded as left ideal of
$B$.
\end{enumerate}
\item[\rm (ii)]An element $x\in M$ is torsion if
there exist $s,n\geq 0$ such that $B_{\geq s}^{\ n}x=0$. The set
of torsion elements of $M$ is denoted by $T(M)$. $M$ is
\textit{torsion} if $T(M)=M$ and \textit{torsion-free} if
$T(M)=0$.
\end{enumerate}
\end{definition}

\begin{theorem}[\cite{lezamalatorre}, Theorem 5.5]\label{theorem16.6.5}
The collection ${\rm stor}-B$ of torsion modules forms a Serre
subcategory of ${\rm sgr}-B$, and the quotient category
\begin{center}
${\rm qsgr}-B:={\rm sgr}-B/{\rm stor}-B$
\end{center}
is abelian.
\end{theorem}

\begin{remark}\label{remark17.4.6}
Recall from the general theory of abelian categories (see
\cite{Smith}) that the quotient functor
\begin{center}
$\pi:{\rm sgr}-B\to {\rm qsgr}-B$
\end{center}
is exact and defined by
\begin{center}
$\pi(M):=M$ and $\pi(f):=\overline{f}$,
\end{center}
with $M$ and $M\xrightarrow{f}N$ in ${\rm sgr}-B$ and the
morphisms in the category ${\rm qsgr}-B$ are defined by
\begin{equation}\label{equation17.4.1}
Hom_{\rm {qsgr}-B}(M,N):=\underrightarrow{\lim}Hom_{{\rm
sgr}-B}(M', N/N'),
\end{equation}
where the direct limit is taken over all $M'\subseteq M$,
$N'\subseteq N$ in ${\rm sgr}-B$ with $M/M'\in {\rm stor}-B$ and
$N'\in {\rm stor}-B$ (see \cite{Grothendieck}, \cite{Gabriel} and
also \cite{Smith} Proposition 2.13.4).
\end{remark}

\begin{definition}[\cite{lezamalatorre}, Definition 6.1]\label{definition16.5.1}
Let $M$ be a semi-graded $B$-module, $M=\bigoplus_{n\in
\mathbb{Z}} M_n$. Let $i\in \mathbb{Z}$, the semi-graded module
$M(i)$ defined by $M(i)_{n}:=M_{i+n}$ is called a shift of $M$,
i.e.,
\begin{center}
$M(i)=\bigoplus_{n\in \mathbb{Z}}M(i)_n=\bigoplus_{n\in
\mathbb{Z}} M_{i+n}$.
\end{center}
\end{definition}

The next proposition shows that the shift of degrees is an
autoequivalence.

\begin{proposition}[\cite{lezamalatorre}, Proposition 6.3]\label{proposition17.5.3}
Let $s:{\rm sgr}-B\to {\rm sgr}-B$ defined by
\begin{center}
$M\mapsto M(1)$,

$M\xrightarrow{f} N\mapsto M(1)\xrightarrow{f(1)} N(1)$,

$f(1)(m):=f(m)$, $m\in M(1)$.
\end{center}
Then,
\begin{enumerate}
\item[\rm (i)]$s$ is an autoequivalence.
\item[\rm (ii)]For every $d\in \mathbb{Z}$, $s^d(M)=M(d)$.
\item[\rm (iii)]$s$ induces an autoequivalence of ${\rm qsgr}-B$ also denoted by $s$.
\end{enumerate}
\end{proposition}

\begin{proposition}\label{proposition17.5.4}
$s\pi=\pi s$.
\end{proposition}
\begin{proof}
We have
\begin{center}
${\rm sgr}-B \xrightarrow{\pi}{\rm qsgr}-B \xrightarrow{s} {\rm
qsgr}-B$ and  ${\rm sgr}-B \xrightarrow{s}{\rm sgr}-B
\xrightarrow{\pi} {\rm qsgr}-B$,
\end{center}
so, for $M$ in ${\rm sgr}-B$, $s\pi(M)=s(\pi(M))=\pi(M)(1)=M(1)$
and $\pi s(M)=\pi(M(1))=M(1)$; for $M\xrightarrow{f}N$ in ${\rm
sgr}-B $, $s\pi(f)=s(\overline{f})=\overline{f(1)}$ and $\pi
s(f)=\pi(f(1))=\overline{f(1)}$.
\end{proof}

\begin{definition}
Let $s$ be the autoequivalence of ${\rm qsgr}-B$ defined by the
shifts of degrees. We define
\begin{center}
$\Gamma(\pi(B))_{\geq 0}:=\bigoplus_{d=0}^\infty Hom_{{\rm
qsgr}-B}(\pi(B),s^d(\pi(B)))$.
\end{center}
\end{definition}

The domain condition on $B$ in Lemma 6.10 of \cite{lezamalatorre}
has been removed in the following simplified version.

\begin{lemma}\label{lemma6.11}
Let $B$ be a ring that satisfies {\rm (C1)-(C4)}.
\begin{enumerate}
\item[\rm (i)]$\Gamma(\pi(B))_{\geq 0}$ is a $\mathbb{N}$-graded ring.
\item[\rm (ii)]Let $\underline{B}:=\bigoplus_{d=0}^\infty Hom_{{\rm sgr}-B}(B,s^d(B))$.
Then, $\underline{B}$ is a $\mathbb{N}$-graded ring and there
exists a ring homomorphism $\underline{B}\to \Gamma(\pi(B))_{\geq
0}$.
\item[\rm (iii)]For any object $M$ of ${\rm sgr}-B$
\begin{center}
$\Gamma(M)_{\geq 0}:=\bigoplus_{d=0}^\infty Hom_{{\rm
sgr}-B}(B,s^d(M))$
\end{center}
is a graded $\underline{B}$-module, and
\begin{center}
$\Gamma(\pi(M))_{\geq 0}:=\bigoplus_{d=0}^\infty Hom_{{\rm
qsgr}-B}(\pi(B),s^d(\pi(M)))$
\end{center}
is a graded $\Gamma(\pi(B))_{\geq 0}$-module.
\item[\rm (iv)]$\underline{B}$ has the following properties:
\begin{enumerate}
\item[\rm (a)]$(\underline{B})_0\cong B_0$ and $\underline{B}$ satisfies {\rm (C2)}.
\item[\rm (b)]$\underline{B}$ satisfies {\rm (C3)}. More generally, let $N$ be a finitely generated graded
$\underline{B}$-module, then every homogeneous component of $N$ is
finitely generated over $(\underline{B})_0$.
\item[\rm (c)]$\underline{B}$ satisfies {\rm (C1)}.
\end{enumerate}
\item[\rm (v)]If $\underline{B}$ satisfies $\mathcal{X}_1$, then
\begin{enumerate}
\item[\rm (a)]$\Gamma(\pi(B))_{\geq 0}$ satisfies {\rm (C2)}.
\item[\rm (b)]$\Gamma(\pi(B))_{\geq 0}$ satisfies {\rm (C3)}. More generally, let $N$ be a finitely generated graded
$\Gamma(\pi(B))_{\geq 0}$-module, then every homogeneous component
of $N$ is finitely generated over $(\Gamma(\pi(B))_{\geq 0})_0$.
\item[\rm (c)]$\Gamma(\pi(B))_{\geq 0}$ satisfies {\rm (C1)}.
\end{enumerate}
\end{enumerate}
\end{lemma}
\begin{proof}
We include only the proof of the part (v) since the proof of the
others are exactly as in \cite{lezamalatorre}.

(v) We set $\Gamma:= \Gamma(\pi(B))_{\geq 0}$. Then,

(a) $\Gamma$ satisfies (C2): From (\ref{equation17.4.1}) we have
$\Gamma_0=Hom_{{\rm qsgr}-B}(\pi(B),\pi(B))=Hom_{{\rm
qsgr}-B}(B,B)=\underrightarrow{\lim}Hom_{{\rm sgr}-B}(I', B/N')$,
where the direct limit is taken over all pairs $(I',N')$ in ${\rm
sgr}-B$, with $I',N'\subseteq B$, $B/I'\in {\rm stor}-B$ and
$N'\in {\rm stor}-B$. Since $\pi$ is a covariant functor, we
obtain a ring homomorphism (taking in particular $I'=B$ and
$N'=0$)
\begin{center}
$(\underline{B})_0=Hom_{{\rm sgr}-B}(B, B)\xrightarrow{\gamma}
Hom_{{\rm qsgr}-B}(B,B)=\Gamma_0$

$\gamma(f):=\pi(f)=\overline{f}$.
\end{center}
Since $B_0\cong (\underline{B})_0$ (isomorphism defined by
$\alpha(x)=\alpha_x,\alpha_x(b):=bx,x\in B_0,b\in B$), then
$\Gamma_0$, $\Gamma$ and $\underline{B}$ are $B_0$-modules.
Actually, they are $B_0$-algebras: We check this for
$\underline{B}$, the proof for $\Gamma(\pi(B))_{\geq 0}$ is
similar, and from this we get also that $\Gamma_0$ is a
$B_0$-algebra. If $f\in Hom_{{\rm sgr}-B}(B,B(n)), g\in Hom_{{\rm
sgr}-B}(B,B(m))$, $x\in B_0$ and $b\in B$, then
\begin{center}
$[x\cdot (f\star g)](b)=x\cdot (s^n(g)\circ f)(b)=xg(n)(f(b))$;

$[f\star(x\cdot g)](b)=[s^n(x\cdot g)\circ f](b)=(x\cdot
g)(n)(f(b))=xg(n)(f(b))$.
\end{center}
Since $B_0$ is noetherian, in order to prove that $\Gamma_0$ is a
noetherian ring, the idea is to show that $\Gamma_0$ is finitely
generated as $B_0$-module, but since $\underline{B}$ satisfies
$\mathcal{X}_1$, this follows from Proposition 3.1.3 (3) in
\cite{Artin2}. Thus, $\Gamma_0$ is a commutative noetherian ring,
and hence, $\Gamma$ satisfies (C2).

(b) $\Gamma$ satisfies (C3): Since $\Gamma$ is graded, $\Gamma_d$
is a $\Gamma_0$-module for every $d$, but by $\gamma$ in (a), the
idea is to prove that $\Gamma_d$ is finitely generated over $B_0$,
but again we apply Proposition 3.1.3 (3) in \cite{Artin2}.

For the second part of (b), let $N$ be a $\Gamma$-module generated
by a finite set of homogeneous elements $x_1,\dots,x_r$, with
$x_i\in N_{d_i}$, $1\leq i\leq r$. Let $x\in N_d$, then there
exist $f_1,\dots,f_r\in \Gamma$ such that $x=f_1\cdot
x_1+\cdots+f_r\cdot x_r$, from this we can assume that $f_i\in
\Gamma_{d-d_i}$, but as was observed before, every
$\Gamma_{d-d_i}$ is finitely generated as $\Gamma_{0}$-module, so
$N_d$ is finitely generated over $\Gamma_{0}$ for every $d$.

(c) $\Gamma$ satisfies (C1): By (iii), $\Gamma$ is not only $SG$
but $\mathbb{N}$-graded. The proof of (C1) is exactly as in the
part (iv) of Lemma 6.10 in \cite{lezamalatorre}.
\end{proof}

\begin{proposition}[\cite{Artin2}, Proposition 2.5]\label{proposition6.12}
Let $S$ be a commutative noetherian ring and $\rho:C\to D$ be a
homomorphism of $\mathbb{N}$-graded left noetherian $S$-algebras.
If the kernel and cokernel of $\rho$ are right bounded, then
$D\underline{\otimes}\,_C\, -$ defines an equivalence of
categories ${\rm qgr}-C\simeq {\rm qgr}-D$, where
$\underline{\otimes}$ denotes the graded tensor product.
\end{proposition}

The solution of Problem 1 is as follows.

\begin{theorem}\label{theorem16.7.13}
Let $B$ be a $SG$ ring that satisfies {\rm (C1)-(C4)} and assume
that $\underline{B}$ satisfies the condition $\mathcal{X}_1$, then
there exists an equivalence of categories
\begin{equation}\label{equation17.5.2}
{\rm qgr}-\underline{B}\simeq{\rm qgr}-\Gamma(\pi(B))_{\geq 0}.
\end{equation}
\end{theorem}
\begin{proof}
The ring homomorphism
\begin{align}\label{equation6.1}
\underline{B} & \xrightarrow{\rho} \Gamma(\pi(B))_{\geq 0} \\
f_{0}+\cdots +f_{d} & \mapsto \pi(f_{0})+\cdots+\pi(f_{d})\notag
\end{align}
satisfies the conditions of Proposition \ref{proposition6.12},
with $S=B_0$, $C=\underline{B}$ and $D=\Gamma(\pi(B))_{\geq 0}$.
In fact, from the proof of Lemma \ref{lemma6.11} we know that
$\underline{B}$ and $\Gamma(\pi(B))_{\geq 0}$ are
$\mathbb{N}$-graded left noetherian $B_0$-algebras. Since
$\underline{B}$ satisfies the condition $\mathcal{X}_1$, we can
apply the proof of part S10 in Theorem 4.5 in \cite{Artin2} to
conclude that the kernel and cokernel of $\rho$ are right bounded.
\end{proof}

We will see next that our Theorem \ref{theorem16.7.13} extends the
Serre-Artin-Zhang-Verevkin equivalence (\ref{equation1.2}).

\begin{corollary}\label{corollary17.5.14}
Let $B$ be a $SG$ ring that satisfies {\rm (C1)-(C4)}. Then,
\begin{enumerate}
\item[\rm (i)]There is an injective homomorphism of $\mathbb{N}$-graded
$B_0$-algebras $\eta:\underline{B}\to Gr(B)$.
\item[\rm (ii)]If $B_0=K$ is a field and $Gr(B)$ is left noetherian and satisfies $\mathcal{X}_1$,
then $\underline{B}$ satisfies $\mathcal{X}_1$ and the following
equivalences of categories hold:
\begin{equation}\label{equation17.5.4a}
{\rm qgr}-Gr(B)\simeq{\rm qgr}-\underline{B}\simeq{\rm
qgr}-\Gamma(\pi(B))_{\geq 0}.
\end{equation}
\item[\rm (iii)]If $B$ is finitely graded and satisfies $\mathcal{X}_1$, then $\underline{B}\cong
B$ and the Serre-Artin-Zhang-Verevkin equivalence ${\rm
qgr}-B\simeq{\rm qgr}-\Gamma(\pi(B))_{\geq 0}$ holds.
\end{enumerate}
\end{corollary}
\begin{proof}
(i) $\eta$ is defined by (see Proposition \ref{proposition17.5.5})
\begin{align*}
\bigoplus_{d=0}^\infty Hom_{{\rm sgr}-B}(B,B(d))=\underline{B} &
\xrightarrow{\eta} Gr(B)= \bigoplus_{d=0}^\infty
Gr(B)_d=\bigoplus_{d=0}^\infty
\frac{B_0\oplus\cdots\oplus B_d}{B_0\oplus\cdots\oplus B_{d-1}}\\
f_0+\cdots+f_d & \mapsto
\overline{f_0(1)}+\cdots+\overline{f_d(1)},
\end{align*}
with $f_i\in Hom_{{\rm sgr}-B}(B,B(i))$, $0\leq i\leq d$. It is
clear thet $\eta$ is additive and $\eta(1)=1$; $\eta$ is
multiplicative:
\begin{center}
$\eta(f_n\star g_m)=\eta(s^n(g_m)\circ
f_n)=\overline{(s^n(g_m)\circ
f_n)(1)}=\overline{s^n(g_m)(f_n(1))}=\overline{g_m(f_n(1))}=
\overline{f_n(1)g_m(1)}=\overline{f_n(1)}\ \overline{
g_m(1)}=\eta(f_n)\eta(g_m)$.
\end{center}
$\eta$ is a $B_0$-homomorphism: Let $x\in B_0$ and $f_d\in
Hom_{{\rm sgr}-B}(B,B(d))$, since $B_0\cong (\underline{B})_0$,
then
\begin{center}
$\eta(x\cdot f_d)=\eta(f_d\circ \alpha_x)=\overline{(f_d\circ
\alpha_x)(1)}=\overline{f_d(\alpha_x(1))}=\overline{f_d(x)}=\overline{x\cdot
f_d(1)}=\overline{x}\cdot \overline{f_d(1)}=x\cdot
\overline{f_d(1)}=x\cdot \eta(f_d)$.
\end{center}

$\eta$ is injective: If
$\overline{f_0(1)}+\cdots+\overline{f_d(1)}=\overline{0}$, then
$\overline{f_k(1)}=0$ for every $0\leq k\leq d$, therefore
$f_k(1)\in (B_0\oplus \cdots \oplus B_{k-1})\cap B_k$  since
$f_k(1)\in B(k)_0=B_k$.

(ii) Since $\underline{B}$ and $Gr(B)$ are $\mathbb{N}$-graded
left noetherian $K$-algebras ($K$ a field) and the kernel and
cokernel of $\eta$ are right bounded, we apply Lemma 8.2 in
\cite{Artin2} to conclude that $\underline{B}$ satisfies
$\mathcal{X}_1$. Thus, from Theorem \ref{theorem16.7.13} we get
the second equivalence of (\ref{equation17.5.4a}). Applying
Proposition \ref{proposition6.12} to $\eta$ we obtain the first
equivalence.

(iii) We define $\theta$ by
\begin{align*}
\bigoplus_{d=0}^\infty Hom_{{\rm sgr}-B}(B,B(d))=\underline{B} &
\xrightarrow{\theta} B=
\bigoplus_{d=0}^\infty B_d\\
f_0+\cdots+f_d & \mapsto f_0(1)+\cdots+f_d(1),
\end{align*}
with $f_i\in Hom_{{\rm sgr}-B}(B,B(i))$, $0\leq i\leq d$. As (i),
we can prove that $\theta$ is an isomorphism of $B_0$-algebras.
Thus, we get the Serre-Artin-Zhang-Verevkin equivalence ${\rm
qgr}-B\simeq{\rm qgr}-\Gamma(\pi(B))_{\geq 0}$.
\end{proof}

\begin{example}\label{example6.14}
In \cite{lezamalatorre} was proved that the following examples of
skew $PBW$ extensions are semi-graded rings (most of them non
$\mathbb{N}$-graded) and satisfy the conditions (C1)-(C4),
moreover, in each case,  $\underline{B}$ satisfies the condition
$\mathcal{X}_1$, therefore, for these algebras Theorem
\ref{theorem16.7.13} is true. In every example $B_0=K$ is a field:
Enveloping algebra of a Lie $K$-algebra $\mathcal{G}$ of dimension
$n$, $\cU(\mathcal{G})$; the quantum algebra
$\mathcal{U}'(so(3,K))$, with $q\in K-\{0\}$; the dispin algebra
$\cU(osp(1,2))$; the Woronowicz algebra
$\cW_{\nu}(\mathfrak{sl}(2,K))$, where $\nu \in K-\{0\}$ is not a
root of unity; eight types of $3$-dimensional skew polynomial
algebras.
\end{example}

\section{Gelfand-Kirillov conjecture}

In this section we will review some aspects of the quantum version
of the Gelfand-Kirillov conjecture and we will formulate a related
problem in the context of the skew $PBW$ extensions. We start
recalling the classical conjecture and some well known cases where
the conjecture has positive answer. In this section $Q(B)$ denotes
the total ring of fraction of an Ore (left and right) domain $B$.

\begin{conjecture}[Gelfand-Kirillov, \cite{GK}]
Let $\mathcal{G}$ be an algebraic Lie algebra of finite dimension
over a field $L$, with $char(L)=0$. Then, there exist integers
$n,k\geq 1$ such that
\begin{equation}\label{GKconjecture}
Q(\mathcal{U}(\mathcal{G}))\cong Q(A_n(L[s_1,\dots,s_k])),
\end{equation}
where $A_n(L[s_1,\dots,s_k])$ is the general Weyl algebra over
$L$.
\end{conjecture}
Recall that $\mathcal{G}$ is \textit{algebraic} if $\mathcal{G}$
is the Lie algebra of a linear affine algebraic group. A group $G$
is \textit{linear affine algebraic} if $G$ is an affine algebraic
variety such that the multiplication and the inversion in $G$ are
morphisms of affine algebraic varieties.

Next we recall some examples of algebraic Lie algebras for which
the classical conjecture (\ref{GKconjecture}) holds.

\begin{example}
(i) The algebra of all $n\times n$ matrices over a field $L$ with
${\rm char}(L)=0$. The same is true for the algebra of matrices of
null trace (\cite{GK}, Lemma 7).

(ii) A finite dimensional nilpotent Lie algebra over a field $L$,
with ${\rm char}(L)=0$. Moreover, in this case,
$Q(Z(\mathcal{U}(\mathcal{G})))\cong
Z(Q(\mathcal{U}(\mathcal{G})))$ (\cite{GK}, Lemma 8).

(iii) A finite dimensional solvable algebraic Lie algebra over the
field $\mathbb{C}$ of complex numbers (see \cite{Joseph}, Theorem
3.2)

(iv) Any algebraic Lie algebra $\mathcal{G}$ over an algebraically
closed field $L$ of characteristic zero, with
$\dim(\mathcal{G})\leq 8$ (see \cite{Alev}).
\end{example}

\begin{remark}
(i) More examples can be found in \cite{Bois}, \cite{Futorny},
\cite{Joseph1} and \cite{Ooms}.

(ii) There are examples of non-algebraic Lie algebras for which
the conjecture is false. However, other examples show that the
conjecture holds for some non-algebraic Lie algebras (see
\cite{GK}, Section 8).
\end{remark}

We are interested in the \textit{quantum version} of the
Gelfand-Kirillov conjecture, i.e., in this case
$\mathcal{U}(\mathcal{G})$ is replaced for a quantum algebra and
the Weyl algebra $A_n(L[s_1,\dots,s_k])$ in (\ref{GKconjecture})
is replaced by a suitable $n$-multiparametric quantum affine space
$K_{\rm \textbf{q}}[x_1,\dots,x_n]$ , as it is shown the following
examples.

\begin{example}[\cite{Cauchon}]
Let $K$ be a field and
$B:=K[x_1][x_2;\sigma_2,\delta_2]\cdots[x_n;\sigma_n,\delta_n]$ be
an iterated skew polynomial ring with some extra adequate
conditions on $\sigma$'s and $\delta$'s. Then there exits
$\textbf{\rm \textbf{q}}:=[q_{i,j}]\in M_n(K)$ with
$q_{ii}=1=q_{ij}q_{ji}$, for every $1\leq i,j\leq n$, such that
$Q(B)\cong Q(K_{\rm \textbf{q}}[x_1,\dots,x_n])$.
\end{example}

\begin{example}[\cite{Alev2}, Theorem 3.5]
Let $A_n^{Q,\Gamma}(K)$ be the multiparameter quantized Weyl
algebra (see \cite{lezamareyes1}); in particular, consider the
case when there exists a parameter $q\in K^*$ that is non root of
unity, such that every parameter in $Q=[q_1,\dots,q_n]$ and
$\Gamma=[\gamma_{ij}]$ is a power of $q$, and in addition,
$q_i\neq 1$, $1\leq i\leq n$. Under these conditions, there exits
$\textbf{\rm \textbf{q}}:=[q_{ij}]\in M_{2n}(K)$ with
$q_{ii}=1=q_{ij}q_{ji}$, and $q_{ij}$ is a power of $q$, $1\leq
i,j\leq n$, such that
\begin{center}
$Q(A_n^{Q,\Gamma}(K))\cong Q(K_{\textbf{q}}[x_1,\dots,x_{2n}])$,

$Z(Q(K_{\textbf{q}}[x_1,\dots,x_{2n}]))=K$.
\end{center}
\end{example}

\begin{example}[\cite{Alev2}, Theorem 2.15]\label{2.5.6}
Let $U_q^{+}(sl_m)$ be the quantum enveloping algebra of the Lie
algebra of strictly superior triangular matrices of size $m\times
m$ over a field $K$.
\begin{enumerate}
\item[\rm (i)]If $m=2n+1$, then
\begin{center}
$Q(U_q^{+}(sl_m))\cong Q({\rm K}_{{\rm q}}[x_1,\dots, x_{2n^2}])$,
\end{center}
where ${\rm K}:=Q(Z(U_q^{+}(sl_m)))$ and ${\rm q}:=[q_{ij}]\in
M_{2n^2}(K)$, with $q_{ii}=1=q_{ij}q_{ji}$, and $q_{ij}$ is a
power of $q$ for every $1\leq i,j\leq 2n^2$.
\item[\rm (ii)]If $m=2n$, then
\begin{center}
$Q(U_q^{+}(sl_m))\cong Q({\rm K}_{{\rm q}}[x_1,\dots,
x_{2n(n-1)}])$,
\end{center}
where ${\rm K}:=Q(Z(U_q^{+}(sl_m)))$ and ${\rm q}:=[q_{ij}]\in
M_{2n(n-1)}(K)$, with $q_{ii}=1=q_{ij}q_{ji}$, and $q_{ij}$ is a
power of $q$ for every $1\leq i,j\leq 2n(n-1)$.
\end{enumerate}
Moreover, in both cases
\begin{center}
$Q(Z(U_q^{+}(sl_m)))\cong Z(Q(U_q^{+}(sl_m)))$.
\end{center}
\end{example}

\subsection{Problem 2}

\begin{itemize}
\item \textit{Study the Gelfand-Kirillov conjecture for bijective skew $PBW$ extensions over Ore
domains. Moreover, investigate for them the isomorphism
$Q(Z(A))\cong Z(Q(A))$}.
\end{itemize}

The ring $Q^{k,n}_{\textbf{q},\sigma}(R)$ of skew quantum
polynomials over $R$, also denoted by
$R_{\textbf{q},\sigma}[x_1^{\pm 1},\dots,x_k^{\pm 1},
x_{k+1},\dots,x_n]$, conforms a particular type of
quasi-commutative bijective skew $PBW$ extension. In
\cite{Lezama-OreGoldie} was proved the following partial solution
of Problem 2.

\begin{theorem}[\cite{Lezama-OreGoldie}, Corollary 5.1]
If $R$ is an Ore domain $($left and right$)$, then
\begin{equation}\label{GKconjectureI}
Q(Q^{k,n}_{\textbf{q},\sigma}(R))\cong
Q(\textbf{Q}_{\textbf{q},\sigma}[x_1,\dots,x_n]),\, \text{with}\,
\textbf{Q}:=Q(R).
\end{equation}
\end{theorem}

The precise definition of $R_{\textbf{q},\sigma}[x_1^{\pm
1},\dots,x_k^{\pm 1}, x_{k+1},\dots,x_n]$ and
$\textbf{Q}_{\textbf{q},\sigma}[x_1,\dots,x_n]$ can be found in
\cite{Lezama-OreGoldie}. Observe that the $n$-multiparametric
quantum affine space $K_{\rm \textbf{q}}[x_1,\dots,x_n]$ was
replaced in (\ref{GKconjectureI}) by the $n$-multiparametric skew
quantum space $\textbf{Q}_{\textbf{q},\sigma}[x_1,\dots,x_n]$.

\section{ Serre's problem}

Given a ring, it is an interesting problem to investigate if the
finitely generated projective modules over it are free. This
problem becomes very important after the formulation in 1955 of
the famous Serre's problem about the freeness of finitely
generated projective modules over the polynomial ring
$K[x_1,\dots,x_n]$, $K$ a field (see \cite{Artamonov2},
\cite{Artamonov3}, \cite{Bass1}, \cite{Lam}). The Serre's problem
was solved positively, and independently, by Quillen in USA, and
by Suslin in Leningrad, USSR (St. Petersburg, Russia) in 1976
(\cite{Quillen}, \cite{Suslin1}).

For arbitrary skew $PBW$ extensions the problem has negative
answer, for example, let $A:=R[x;\sigma]$ be the skew polynomial
ring over $R:=K[y]$, where $K$ is a field and $\sigma(y):=y+1$,
from \cite{McConnell}, 12.2.11 we can conclude that there exist
finitely generated projective modules over $A$ that are not free.
However, in \cite{ArtamonovLezamaFajardo} was proved that if $K$
is a field, $A:=K[x_1,\dots,x_n;\sigma]$, $\sigma$ is an
automorphism of $K$ of finite order, $x_ix_j=x_jx_i$ and
$x_ir=\sigma(r)x_i$, for every $r\in K$ and $1\leq i,j\leq n$,
then every finitely generated projective module over $A$ is free.
The proof in \cite{ArtamonovLezamaFajardo} of this theorem
(\textit{Quillen-Suslin theorem}) is not construcitve, i.e., the
proof shows the existence of a basis for every finitely generated
projective module $M$ over $A$, but an explicit basis of $M$ was
not constructed.

\subsection{Problem 3}

\begin{itemize}
\item \textit{Give a matrix-constructive proof of the Quillen-Suslin
theorem for $A:=K[x_1,\dots,x_n;\sigma]$. Include an algorithm
that computes a basis of a given finitely generated projective
$A$-module.}
\end{itemize}

In \cite{FajardoLezama} was solved the previous problem for the
case of only one variable, in this particular situation weaker
hypotheses can be set, namely, $K$ is a division ring, $\sigma$ is
not necessarily of finite order and a non trivial
$\sigma$-derivation $\delta$ of $K$ can be added.

\begin{theorem}[\cite{FajardoLezama}, Theorem 2.1]\label{theorem9.6.1}
Let $K$ be a division ring and $A:=K[x;\sigma,\delta]$, with
$\sigma$ bijective. Then $A$ is a $\mathcal{PF}$-ring, i.e., every
finitely generated projective module over $A$ is free.
\end{theorem}

The proof of the previous theorem given in \cite{FajardoLezama} is
not only matrix-constructive, but also some algorithms that
compute the basis of a given finitely generated projective
$A$-module are exhibited and implemented in \texttt{Maple}. Next
we review the main ingredients of the proof and include an
illustrative example of algorithms.

\begin{proposition}[\cite{Gallego}]\label{6.2.4}
Let $B$ be a ring. $B$ is a $\mathcal{PF}$-ring if and only if for
every $s\geq 1$, given an idempotent matrix $F\in M_s(B)$, there
exists an invertible matrix $U\in GL_s(B)$ such that
\begin{equation}\label{eq6.2.4}
UFU^{-1}=\begin{bmatrix}0 & 0\\
0 & I_r \end{bmatrix},
\end{equation}
where $r=dim(\langle F\rangle)$, $0\leq r\leq s$, and $\langle
F\rangle$ represents the free left $B$-module $M$ generated by the
rows of $F$. Moreover, the final $r$ rows of $U$ form a basis of
$M$.
\end{proposition}

The matrix-constructive proof of Theorem \ref{theorem9.6.1}
consists in constructing explicitly the matrix $U$ in Proposition
\ref{6.2.4} from the entries of the given idempotent matrix $F$
(recall that a left module $M$ over a ring $B$ is finitely
generated projective if and only if $M$ is the left $B$-module
generated by the rows of a idempotent matrix over $B$). In
\cite{FajardoLezama} were designed two algorithms for calculating
the basis of a given finitely generated projective module: a
constructive simplified version and a more complete computational
version over a field. The computational version was implemented
using ${\rm Maple}^\circledR$ 2016. For completeness, we include
next the constructive simplified version (see
\cite{FajardoLezama}, Algorithm 1):

\begin{center}{\small
\fbox{\parbox[c]{12cm}{
\begin{center}
{\rm \textbf{Algorithm for the Quillen-Suslin theorem:\\
Constructive version}}
\end{center}
\begin{description}
\item[]{\rm \textbf{INPUT}:} $A:=K[x,\sigma,\delta]$; $F\in M_s(A)$ an idempotent
matrix.
\item[]{\rm \textbf{OUTPUT}:} Matrices $U$, $U^{-1}$ and a
basis $X$ of $\langle F\rangle$, where
\begin{equation}\label{equation14.6.1}
UFU^{-1}=\begin{bmatrix}0 & 0\\
0 & I_r \end{bmatrix} \ \text{and}\ r=dim(\langle F\rangle).
\end{equation}
\item[]{\rm \textbf{INITIALIZATION}:} $F_1:=F$.

{\rm \textbf{FOR}} $k$ from $1$ to $n-1$ {\rm \textbf{DO}}
\begin{enumerate}
\item Follow the reduction procedures
(B1) and (B2) in the proof of Theorem \ref{theorem9.6.1} in order
to compute matrices $U_k'$, $U_k'^{-1}$ and $F_{k+1}$ such that
\begin{equation*}
U_k'F_kU_k'^{-1}=\begin{bmatrix}\alpha_k & 0\\
0 & F_{k+1} \end{bmatrix},\text{ where } \alpha_k\in\{0,1\}.
\end{equation*}
\item $U_{k}:=\begin{bmatrix}I_{k-1} & 0\\
0 & U_{k}'\end{bmatrix}U_{k-1}$; compute $U_k^{-1}$.
\item By permutation matrices modify $U_{n-1}$.
\end{enumerate}
{\rm \textbf{RETURN}} $U:=U_{n-1}$, $U^{-1}$ satisfying
(\ref{equation14.6.1}), and a basis $X$ of $\langle F\rangle$.
\end{description}}}}
\end{center}

\begin{example}(\cite{FajardoLezama}, Example 3.4)
Let $A:=K[x,\sigma,\delta]$, $K:=\mathbb{Q}(t)$,
$\sigma(f(t)):=f(qt)$ and
$\delta(f(t)):=\tfrac{f(qt)-f(t)}{t(q-1)}$, where $q\in
K-\{0,1\}$. We consider the idempotent matrix $F:=[F^{(1)} \
F^{(2)} \ F^{(3)}\ F^{(4)}]\in M_4(A)$, with $F^{(i)}$ the $ith$
column of $F$ and $a\in\mathbb{Q}$, where
\[
F^{(1)}=
\begin{bmatrix}
-t^{2}qx^2\\
\left(-ta+2\,t \right)x-2\,a+2\\
tx+2\\
-1
\end{bmatrix},
\]

\[
F^{(2)}=
\begin{bmatrix}
-2\,tx+2\\
-{t}^{2}q{x}^{2}+ \left( ta-4\,t \right) x+2\,a-1\\
-tx-2\\
tx+2
\end{bmatrix},
\]

\[
F^{(3)}=
\begin{bmatrix}
-tx-2\\
\left( -2\,{t}^{2}qa+3\,{t}^{2}q \right) {x}^{2}+ \left( {a}^{2}t-8\,ta+8\,t \right) x+2\,{a}^{2}-3\,a+1\\
{t}^{2}q{x}^{2}+ \left( -ta+4\,t \right) x-2\,a+2\\
\left( ta-2\,t \right) x+2\,a-2
\end{bmatrix},
\]

\[
F^{(4)}=
\begin{bmatrix}
-{t}^{3}{q}^{3}{x}^{3}+\left( -{q}^{2}{t}^{2}-5\,{t}^{2}q \right) {x}^{2}-5\,tx+2\\
-{t}^{3}{q}^{3}{x}^{3}+\left( -{q}^{2}{t}^{2}-3\,{t}^{2}q \right) {x}^{2}+ \left( -ta+t \right) x-2\,a+2\\
tx+2\\
{t}^{2}q{x}^{2}+2\,tx-1
\end{bmatrix}.
\]
Applying the algorithms we obtain
\[
U =
\begin{bmatrix}
tx+1 & 0 & {t}^{2}q{x}^{2}+2\,tx-1 & {t}^{2}q{x}^{2}+3\,tx\\
1 & -tx-2 & \left( -ta+2\,t \right) x-2\,a+2 & -{t}^{2}q{x}^{2}-2\,tx+2\\
tx-1 & 1 & {t}^{2}q{x}^{2}+a-1 & {t}^{2}q{x}^{2}+2\,tx-1\\
1 & 0 & tx & tx+1
\end{bmatrix},
\]
{\tiny
\[
U^{-1}=
\begin{bmatrix}
tx & -1 & -tx-2 & 0\\
a-1 & -tx+a-1 & -{t}^{2}q{x}^{2}+ \left( ta-4\,t \right) x+2\,a-1
& {t}^{3}{q}^{3}{x}^{3}-
\left( -q+a-4 \right) {t}^{2}q{x}^{2}+ \left( -3\,ta+3\,t \right) x+1\\
-1 & -1 & -tx-2 & {t}^{2}q{x}^{2}+3\,tx\\
0 & 1 & tx+2 & -{t}^{2}q{x}^{2}-2\,tx+1
\end{bmatrix},
\]}
\[
UFU^{-1} =
\begin{bmatrix}
0 & 0 & 0 & 0\\
0 & 0 & 0 & 0\\
0 & 0 & 1 & 0\\
0 & 0 & 0 & 1
\end{bmatrix},
\]

Therefore, $r=2$ and the last two rows of $U$ conform a basis
$X=\{\emph{\textbf{x}}_1,\textbf{\emph{x}}_2\}$, of $\langle
F\rangle$,
\begin{center}
$\emph{\textbf{x}}_1=(tx-1,1,{t}^{2}q{x}^{2}+a-1,{t}^{2}q{x}^{2}+2\,tx-1)$,
$\emph{\textbf{x}}_2=(1,0,tx,tx+1)$.
\end{center}

\end{example}

\begin{remark}
For $n\geq 2$, the commutativity of $K$ in Problem 3 is necessary,
in \cite{Lam} (p. 36) is proved that if $B$ is a division ring,
then $B[x,y]$ is not $\mathcal{PF}$.
\end{remark}

\section{Artin-Schelter open problems}

Artin-Schelter regular algebras (shortly denoted as $AS$) were
introduced by Michael Artin and William Schelter in \cite{Artin1}.
In non-commutative algebraic geometry these algebras play the role
of $K[x_1,\dots,x_n]$ in commutative algebraic geometry, thus, in
particular, the noetherian $AS$ algebras satisfy the condition
$\mathcal{X}$ of Subsection \ref{subsection1.1} (see Theorem 12.6
in \cite{Smith}), and hence, for them the
Serre-Artin-Zhang-Verevkin theorem holds. Nowadays $AS$ algebras
are intensively investigated, in the next subsection we present
the definition and some properties and open problems on these
algebras.

\subsection{Artin-Schelter regular algebras}

\begin{definition}[\cite{Artin1}]\label{definitionB.2.1}
Let $K$ be a field and $A$ be a finitely graded $K$-algebra. It
says that $A$ is an Artin-Schelter regular algebra $(AS)$ if
\begin{enumerate}
\item[\rm (i)]${\rm gld}(A)=d<\infty$.
\item[\rm (ii)]${\rm GKdim}(A)<\infty$.
\item[\rm (iii)]$\underline{Ext}_A^i(_AK,_AA)\cong \begin{cases}
0\ \  if \ i\neq d\\
K(l)_A\ \ if \ i=d
\end{cases}$
\end{enumerate}
for some shift $l\in \mathbb{Z}$.
\end{definition}

\begin{remark}\label{remark17.3.2}
(i) In \cite{Rogalski} is showed that the third condition is
equivalent to
\begin{center}
$\underline{Ext}_A^i(K_A,A_A)\cong \begin{cases}
0\ \  if \ i\neq d\\
_AK(l)\ \ if \ i=d.
\end{cases}$
\end{center}
In addition, since $K\cong A/A_{\geq 1}$ is finitely generated as
left $A$-module, then we can replace
$\underline{Ext}_A^i(_AK,_AA)$ by $Ext_A^i(_AK,_AA)$. The same is
true for $\underline{Ext}_A^i(K_A,A_A)$.

(ii) Some key examples of Artin-Schelter regular algebras are: Let
$A = K[x_1,\dots, x_n]$ be the usual commutative polynomial ring
in $n$ variables over the field $K$, with any weights
$\deg(x_i)=d_i\geq 1$, $1\leq i\leq n$, then $A$ is $AS$; if $A$
is a commutative $AS$ algebra, then $A$ is isomorphic to a
weighted commutative polynomial ring; the Jordan and quantum
planes are $AS$ of dimension $2$; more generally, the
$n$-multiparametric quantum affine space $K_{\rm
\textbf{q}}[x_1,\dots,x_n]$ is $AS$. $AS$ algebras of dimension
$\leq 3$ are completely classified. It is an interesting and
intensively investigated open problem to classify the $AS$
algebras of dimension $4$ or $5$.

(iii) \textit{Artin-Schelter open problems}: In \cite{Artin1} were
formulated the following problems:
\begin{enumerate}
\item Any $AS$ algebra is noetherian?
\item Any $AS$ algebra is a domain?
\end{enumerate}
There are some advances with respect to the previous open
problems: Both questions above have positive answer for any $AS$
algebra of dimension $\leq 3$. In addition, the same is true for
all known concrete examples of $AS$ algebras of higher dimension.
If $A$ is $AS$ noetherian and its dimension is $\leq 4$, then $A$
is a domain.
\end{remark}

\subsection{Semi-graded Artin-Schelter regular algebras}

$AS$ algebras are $\mathbb{N}$-graded and connected, recently
Gaddis in \cite{Gaddis} and \cite{Gaddis2} introduces the
technique of \textit{homogenization} and the notion of
\textit{essentially regular algebras} in order the study of the
Artin-Schelter condition for non $\mathbb{N}$-graded algebras. On
the other hand, with the purpose of giving new examples of $AS$
algebras, in \cite{Lu} and \cite{Lu1} are defined the
\textit{$\mathbb{Z}^s$-graded Artin-Schelter regular algebras} and
there in have been proved some results for the classification of
$AS$ algebras of dimension $5$ with two generators. The
\textit{semi-graded Artin-Schelter regular algebras} that we will
introduce in the present section is a new different approach to
this problem, and extend the classical notion of Artin-Schelter
regular algebra defined originally in \cite{Artin1}. In a
forthcoming paper we will generalize some classical well-known
results on Artin-Schelter regular algebras to the semi-graded
case.

\begin{definition}\label{definition18.5.1}
Let $K$ be a field and $B$ be a $K$-algebra. We say that $B$ is a
left semi-graded Artin-Schelter regular algebra $(\mathcal{SAS})$
if the following conditions hold:
\begin{enumerate}
\item[\rm (i)]$B$ is a $FSG$ ring with semi-graduation $B=\bigoplus_{p\geq 0}B_p$.
\item[\rm (ii)]$B$ is connected, i.e., $B_0=K$.
\item[\rm (iii)]${\rm lgld}(B):=d<\infty$.
\item[\rm (iv)]$Ext_{B}^i(_B(B/B_{\geq 1}),_BB)\cong \begin{cases}
0\ \  if \ i\neq d\\
(B/B_{\geq 1})_B\ \ if \ i=d
\end{cases}$
\end{enumerate}
\end{definition}

\begin{remark}
(i) If $K\cap B_{\geq 1}\neq 0$, then $B/B_{\geq 1}=0$; if $K\cap
B_{\geq 1}=0$ then $B/B_{\geq 1}\cong K$, where this is an
isomorphism of $K$-algebras induced by the canonical projection
$\epsilon:B\to K$ (recall that $B_{\geq 1}$ is semi-graded).
Moreover, $_B(B/B_{\geq 1})\cong _BK$ and $(B/B_{\geq 1})_B\cong
K_B$.

(ii) Our definition extends (except for the ${\rm GKdim}$) the
classical notion of Artin-Schelter algebra since in such case
$B/B_{\geq 1}\cong K$, ${\rm lgld}(B)={\rm pd}(_BK)={\rm
pd}(K_B)={\rm rgld}(B)={\rm gld}(B)$ and
$Ext_{B}^i(K,B)=\underline{Ext}_{B}^i(K,B)$ (see \cite{Rogalski}).
Thus, every Artin-Schelter regular algebra is $\mathcal{SAS}$.
\end{remark}

\subsection{Problem 4}

\begin{itemize}
\item Is any $\mathcal{SAS}$ algebra with finite ${\rm GKdim}$
a left noetherian domain?
\end{itemize}

We will present next some examples of $\mathcal{SAS}$ algebras
that are not Artin-Schelter, in every example, the algebra is a
left-right noetherian domain.

For the matrix representation of homomorphisms of left modules in
this section we will use the left row notation and for right
modules the right column notation. In order to compute free
resolutions we will apply Theorem 19 of \cite{Lezama3} that has
been implemented in the library \texttt{SPBWE.lib} developed in
\texttt{Maple} by W. Fajardo in \cite{Fajardo2} (see also
\cite{Fajardo3}). This library contains the packages
\texttt{SPBWETools}, \texttt{SPBWERings} and \texttt{SPBWEGrobner}
with utilities to define and perform calculations with skew $PBW$
extensions.

\begin{example}\label{example18.5.3}
The Weyl algebra $A_1(K)$ is not Artin-Schelter, but it is a
$\mathcal{SAS}$ algebra. In fact, the canonical semi-graduation of
$A_1(K)$ is
\begin{center}
$A_1(K)=K\oplus\, _K\langle x,y\rangle\oplus\, _K\langle
x^2,xy,y^2\rangle\oplus \cdots $
\end{center}
Recall that ${\rm gld}(A_1(K))\leq 2$ (\cite{McConnell});
moreover, $A_1(K)_{\geq 1}=A_1(K)$, so $A_1(K)/A_1(K)_{\geq 1}=0$.
Hence, the condition (iv) in Definition \ref{definition18.5.1}
trivially holds. This argument can be applied also to the general
Weyl algebra $A_n(K)$ and the generalized Weyl algebra $B_n(K)$.
\end{example}

\begin{example}
The quantum deformation $A_1^q(K)$ of the Weyl algebra is not
Artin-Schelter, but it is a $\mathcal{SAS}$ algebra. Indeed,
recall that $A_1^q(K)$ is the $K$-algebra defined by the relation
$yx=qxy+1$, with $q\in K-\{0\}$ ($A_1^q(K)$ coincides with the
additive analog of the Weyl algebra in two variables, as well as,
with the linear algebra of $q$-differential operators in two
variables). The semi-graduation of $A_1^q(K)$ is as in the
previous example, moreover ${\rm gld}(A_1^q(K))\leq 2$ and
$A_1^q(K)_{\geq 1}=A_1^q(K)$, so $A_1^q(K)/A_1^q(K)_{\geq 1}=0$
and the condition (iv) in Definition \ref{definition18.5.1} holds.
\end{example}

\begin{example}
Consider the algebra $\mathcal{J}_1:=K\langle x,y\rangle/\langle
yx-xy+y^2+1\rangle$, according to Corollary 2.3.14 in
\cite{Gaddis}, ${\rm gld}(\mathcal{J}_1)=2$. The semi-graduation
of $\mathcal{J}_1$ is as in Example \ref{example18.5.3}, so
$(\mathcal{J}_1)_{\geq 1}=\mathcal{J}_1$. Thus, $\mathcal{J}_1$ is
$\mathcal{SAS}$ but is not Artin-Schelter.
\end{example}

\begin{example}\label{example18.4.5}
Now we consider the dispin algebra $\cU(osp(1,2))$, which in this
example we will be denoted simply as $B$ (see Example
\ref{example6.14}). Recall that $B$ is defined by the relations
\begin{center}
$x_1x_2-x_2x_1=x_1,\ \ \ x_3x_1+x_1x_3=x_2,\ \ \
x_2x_3-x_3x_2=x_3$.
\end{center}
Since $B$ is not finitely graded, then $B$ is not Artin-Schelter.
We will show that $B$ is $\mathcal{SAS}$. We know that ${\rm
gld}(B)=3$ (see \cite{lezamareyes1}); the semi-graduation of $B$
is given by
\begin{center}
$B=K\oplus\, _K\langle x_1,x_2,x_3\rangle\oplus\, _K\langle
x_1^2,x_1x_2,x_1x_3,x_2^2,x_2x_3,x_3^2\rangle\oplus \cdots $
\end{center}
Note that $B_{\geq 1}=\oplus_{p\geq 1}B_p$ and this ideal
coincides with the two-sided ideal of $B$ generated by
$x_1,x_2,x_3$; moreover, $_B(B/B_{\geq 1})\cong _BK$, where the
structure of left $B$-module for $K$ is given by the canonical
projection $\epsilon:B\to K$ (the same is true for the right
structure, $(B/B_{\geq 1})_B\cong K_B$). With \texttt{SPBWE} we
get the following free resolution of $_BK$: {\small
\begin{equation*}
0\rightarrow B\xrightarrow{\phi_2=\begin{bmatrix}-x_3 & x_2 & x_1
\end{bmatrix}} B^3 \xrightarrow{\phi_1=\begin{bmatrix}1+x_2 & -x_1 & 0
\\ x_3 & -1 & x_1\\
0 & x_3 & 1-x_2\end{bmatrix}}B^3
\xrightarrow{\phi_0=\begin{bmatrix}x_1\\
x_2\\
x_3\end{bmatrix}}B\xrightarrow{\epsilon}K\to 0.
\end{equation*}}
Now we apply $Hom_B(-,\, _BB)$ and we get the complex of right
$B$-modules {\footnotesize
\begin{equation*}
0\rightarrow Hom_B(K,B)\xrightarrow{\epsilon^*} Hom_B(B,B)
\xrightarrow{\phi_0^*}Hom_{B}(B^3,B)
\xrightarrow{\phi_1^*}Hom_{B}(B^3,B)\xrightarrow{\phi_2^*}Hom_{B}(B,B)\to
0.
\end{equation*}}
Note that $Hom_B(K,B)=0$: In fact, let $\alpha\in Hom_B(K,B)$ and
$\alpha(1):=b\in B$, then
$\alpha(x_11)=\alpha(0)=0=x_1\alpha(1)=x_1b$, so $b=0$ (recall
that $B$ is a domain) and from this $\alpha(k)=0$ for every $k\in
K$, i.e., $\alpha=0$. Moreover, from the isomorphisms of right
$B$-modules $Hom_B(B,B)\cong B$ and $Hom_{B} (B^3, B)\cong B^3$ we
obtain the complex
\begin{equation*}
0\rightarrow B\xrightarrow{\phi_0^*=\begin{bmatrix}x_1
\\ x_2\\
x_3\end{bmatrix}}B^3 \xrightarrow{\phi_1^*=\begin{bmatrix}1+x_2 &
-x_1 & 0\\
x_3 & -1 & x_1\\
0 & x_3 & 1-x_2
\end{bmatrix}}B^3\xrightarrow{\phi_2^*=\begin{bmatrix}-x_3 & x_2 & x_1\end{bmatrix}}B\to
0.
\end{equation*}
So, $Ext_{B}^0(K,B)=Hom_B(K,B)=0$,
$Ext_{B}^1(K,B)=0=Ext_{B}^2(K,B)$ and
$Ext_{B}^3(K,B)=B/Im(\phi_2^*)=B/B_{\geq 1}\cong K_{B}$. This
shows that $B$ is $\mathcal{SAS}$.
\end{example}

\begin{example}
The next examples are similar to the previous, in every case ${\rm
gld}(B)=3$ and $B$ is $\mathcal{SAS}$. The free resolutions have
been computed with \texttt{SPBWE}.

(a) Consider the universal enveloping algebra of the Lie algebra
$\mathfrak{sl}(2,K)$, $B:=\mathcal{U}(\mathfrak{sl}(2,K))$. $B$ is
the $K$-algebra generated by the variables $x,y,z$ subject to the
relations
\begin{equation*}
[x,y]=z, \ \ \ [x,z]=-2x,\ \ \ [y,z]=2y.
\end{equation*}
The free resolutions are: {\small
\begin{equation*}
0\rightarrow B\xrightarrow{\phi_2=\begin{bmatrix}-z & y & -x
\end{bmatrix}} B^3 \xrightarrow{\phi_1=\begin{bmatrix}y & -x & 1
\\ z-2 & 0 & -x\\
0 & z+2 & -y\end{bmatrix}}B^3
\xrightarrow{\phi_0=\begin{bmatrix}x\\
y\\
z\end{bmatrix}}B\xrightarrow{\epsilon}K\to 0.
\end{equation*}}
\begin{equation*}
0\rightarrow B\xrightarrow{\phi_0^*=\begin{bmatrix}x
\\ y\\
z\end{bmatrix}}B^3 \xrightarrow{\phi_1^*=\begin{bmatrix}y & -x & 1
\\ z-2 & 0 & -x\\
0 & z+2 &
-y\end{bmatrix}}B^3\xrightarrow{\phi_2^*=\begin{bmatrix}-z & y &
-x
\end{bmatrix}}B\to
0.
\end{equation*}
(b) Now let $B:=\mathcal{U}(\mathfrak{so}(3,K))$ be the
$K$-algebra generated by the variables $x,y,z$ subject to the
relations
\[
[x,y]=z,\ \ \ \ [x,z]=-y,\ \ \ \ \ [y,z]=x.
\]
In this case we have {\small
\begin{equation*}
0\rightarrow B\xrightarrow{\phi_2=\begin{bmatrix}-z & y & -x
\end{bmatrix}} B^3 \xrightarrow{\phi_1=\begin{bmatrix}y & -x & 1
\\ z & -1 & -x\\
1 & z & -y\end{bmatrix}}B^3
\xrightarrow{\phi_0=\begin{bmatrix}x\\
y\\
z\end{bmatrix}}B\xrightarrow{\epsilon}K\to 0,
\end{equation*}}

\begin{equation*}
0\rightarrow B\xrightarrow{\phi_0^*=\begin{bmatrix}x
\\ y\\
z\end{bmatrix}}B^3 \xrightarrow{\phi_1^*=\begin{bmatrix}y & -x & 1
\\ z & -1 & -x\\
1 & z & -y\end{bmatrix}}B^3\xrightarrow{\phi_2^*=\begin{bmatrix}-z
& y & -x
\end{bmatrix}}B\to
0.
\end{equation*}
(c) Quantum algebra $B:=\mathcal{U}'(so(3,K))$, with $q\in
K-\{0\}$:
\begin{center}
$x_2x_1-qx_1x_2=-q^{1/2}x_3,\ \ \
x_3x_1-q^{-1}x_1x_3=q^{-1/2}x_2,\ \ \ x_3x_2-qx_2x_3=-q^{1/2}x_1$.
\end{center}
In this case the free resolutions are: {\small
\begin{equation*}
0\rightarrow B\xrightarrow{\phi_2=\begin{bmatrix}-x_3 & x_2 & -x_1
\end{bmatrix}} B^3 \xrightarrow{\phi_1=\begin{bmatrix}x_2 & -qx_1 &
q^{1/2}
\\ qx_3 & -q^{1/2} & -x_1\\
q^{1/2} & x_3 & -qx_2\end{bmatrix}}B^3
\xrightarrow{\phi_0=\begin{bmatrix}x_1\\
x_2\\
x_3\end{bmatrix}}B\xrightarrow{\epsilon}K\to 0,
\end{equation*}}

\begin{equation*}
0\rightarrow B\xrightarrow{\phi_0^*=\begin{bmatrix}x_1
\\ x_2\\
x_3\end{bmatrix}}B^3 \xrightarrow{\phi_1^*=\begin{bmatrix}x_2 &
-qx_1 & q^{1/2}
\\ qx_3 & -q^{1/2} & -x_1\\
q^{1/2} & x_3 &
-qx_2\end{bmatrix}}B^3\xrightarrow{\phi_2^*=\begin{bmatrix}-x_3 &
x_2 & -x_1
\end{bmatrix}}B\to
0.
\end{equation*}

 (d) Woronowicz algebra
$B:=\cW_{\nu}(\mathfrak{sl}(2,K))$, where $\nu \in K-\{0\}$ is not
a root of unity:
\begin{center}
$x_1x_3-\nu^4x_3x_1=(1+\nu^2)x_1,\ \ \ x_1x_2-\nu^2x_2x_1=\nu
x_3,\ \ \ x_3x_2-\nu^4x_2x_3=(1+\nu^2)x_2$.
\end{center}
The free resolutions are: {\tiny
\begin{equation*}
0\rightarrow B\xrightarrow{\phi_2=\begin{bmatrix}-\nu^4x_3 &
\nu^6x_2 & -x_1
\end{bmatrix}} B^3 \xrightarrow{\phi_1=\begin{bmatrix}\nu^2x_2 & -x_1 &
\nu
\\ \nu^4x_3+(\nu^2+1) & 0 & -x_1\\
0 & x_3-(\nu^2+1) & -\nu^4x_2\end{bmatrix}}B^3
\xrightarrow{\phi_0=\begin{bmatrix}x_1\\
x_2\\
x_3\end{bmatrix}}B\xrightarrow{\epsilon}K\to 0,
\end{equation*}}
{\tiny
\begin{equation*}
0\rightarrow B\xrightarrow{\phi_0^*=\begin{bmatrix}x_1
\\ x_2\\
x_3\end{bmatrix}}B^3 \xrightarrow{\phi_1^*=\begin{bmatrix}\nu^2x_2
& -x_1 & \nu
\\ \nu^4x_3+(\nu^2+1) & 0 & -x_1\\
0 & x_3-(\nu^2+1) &
-\nu^4x_2\end{bmatrix}}B^3\xrightarrow{\phi_2^*=\begin{bmatrix}-\nu^4x_3
& \nu^6x_2 & -x_1
\end{bmatrix}}B\to
0.
\end{equation*}}
\end{example}

In the following example we study the semi-graded Artin-Schelter
condition for the eight types of $3$-dimensional skew polynomial
algebras considered in Example \ref{example6.14}, five of them are
$\mathcal{SAS}$ and the other three are not.

\begin{example}
The first type coincides with the dispin algebra taking
$\beta=-1$, and the fifth type corresponds to
$\mathcal{U}(\mathfrak{so}(3,K))$, thus they are $\mathcal{SAS}$.
For the other six types we compute next the free resolutions with
\texttt{SPBWE}:
\begin{center}
$x_2x_3-x_3x_2=0,\ \ x_3x_1-\beta x_1x_3=x_2,\ \ x_1x_2-x_2x_1=0$
($\mathcal{SAS}$):

{\small
\begin{equation*}
0\rightarrow B\xrightarrow{\phi_2=\begin{bmatrix}-x_3 & x_2 &
-\beta x_1
\end{bmatrix}} B^3 \xrightarrow{\phi_1=\begin{bmatrix}x_2 & -x_1 & 0
\\ x_3 & -1 & -\beta x_1\\
0 & x_3 & -x_2\end{bmatrix}}B^3
\xrightarrow{\phi_0=\begin{bmatrix}x_1\\
x_2\\
x_3\end{bmatrix}}B\xrightarrow{\epsilon}K\to 0,
\end{equation*}}

\begin{equation*}
0\rightarrow B\xrightarrow{\phi_0^*=\begin{bmatrix}x_1
\\ x_2\\
x_3\end{bmatrix}}B^3 \xrightarrow{\phi_1^*=\begin{bmatrix}x_2 &
-x_1 & 0
\\ x_3 & -1 & -\beta x_1\\
0 & x_3 &
-x_2\end{bmatrix}}B^3\xrightarrow{\phi_2^*=\begin{bmatrix}-x_3 &
x_2 & -\beta x_1\end{bmatrix}}B\to 0.
\end{equation*}

$x_2x_3-x_3x_2=x_3,\ \ x_3x_1-\beta x_1x_3=0,\ \
x_1x_2-x_2x_1=x_1$ ($\mathcal{SAS}$):

{\footnotesize
\begin{equation*}
0\rightarrow B\xrightarrow{\phi_2=\begin{bmatrix}-x_3 & x_2 &
-\beta x_1
\end{bmatrix}} B^3 \xrightarrow{\phi_1=\begin{bmatrix}x_2+1 & -x_1 & 0
\\ x_3 & 0 & -\beta x_1\\
0 & x_3 & -x_2+1\end{bmatrix}}B^3
\xrightarrow{\phi_0=\begin{bmatrix}x_1\\
x_2\\
x_3\end{bmatrix}}B\xrightarrow{\epsilon}K\to 0,
\end{equation*}}

{\footnotesize
\begin{equation*}
0\rightarrow B\xrightarrow{\phi_0^*=\begin{bmatrix}x_1
\\ x_2\\
x_3\end{bmatrix}}B^3 \xrightarrow{\phi_1^*=\begin{bmatrix}x_2+1 &
-x_1 & 0
\\ x_3 & 0 & -\beta x_1\\
0 & x_3 &
-x_2+1\end{bmatrix}}B^3\xrightarrow{\phi_2^*=\begin{bmatrix}-x_3 &
x_2 & -\beta x_1\end{bmatrix}}B\to 0.
\end{equation*}}

$x_2x_3-x_3x_2=x_3,\ \ x_3x_1-\beta x_1x_3=0,\ \ x_1x_2-x_2x_1=0$
(not $\mathcal{SAS}$) :

{\footnotesize
\begin{equation*}
0\rightarrow B\xrightarrow{\phi_2=\begin{bmatrix}-x_3 & x_2-1 &
-\beta x_1
\end{bmatrix}} B^3 \xrightarrow{\phi_1=\begin{bmatrix}x_2 & -x_1 & 0
\\ x_3 & 0 & -\beta x_1\\
0 & x_3 & -x_2+1\end{bmatrix}}B^3
\xrightarrow{\phi_0=\begin{bmatrix}x_1\\
x_2\\
x_3\end{bmatrix}}B\xrightarrow{\epsilon}K\to 0,
\end{equation*}}

{\footnotesize
\begin{equation*}
0\rightarrow B\xrightarrow{\phi_0^*=\begin{bmatrix}x_1
\\ x_2\\
x_3\end{bmatrix}}B^3 \xrightarrow{\phi_1^*=\begin{bmatrix}x_2 &
-x_1 & 0
\\ x_3 & 0 & -\beta x_1\\
0 & x_3 &
-x_2+1\end{bmatrix}}B^3\xrightarrow{\phi_2^*=\begin{bmatrix}-x_3 &
x_2-1 & -\beta x_1\end{bmatrix}}B\to 0.
\end{equation*}}

$x_2x_3-x_3x_2=0,\ \ x_3x_1-x_1x_3=0,\ \ x_1x_2-x_2x_1=x_3$
($\mathcal{SAS}$):

{\footnotesize
\begin{equation*}
0\rightarrow B\xrightarrow{\phi_2=\begin{bmatrix}-x_3 & x_2 & -x_1
\end{bmatrix}} B^3 \xrightarrow{\phi_1=\begin{bmatrix}x_2 & -x_1 &
1
\\ x_3 & 0 & -x_1\\
0 & x_3 & -x_2\end{bmatrix}}B^3
\xrightarrow{\phi_0=\begin{bmatrix}x_1\\
x_2\\
x_3\end{bmatrix}}B\xrightarrow{\epsilon}K\to 0,
\end{equation*}}

{\footnotesize
\begin{equation*}
0\rightarrow B\xrightarrow{\phi_0^*=\begin{bmatrix}x_1
\\ x_2\\
x_3\end{bmatrix}}B^3 \xrightarrow{\phi_1^*=\begin{bmatrix}x_2 &
-x_1 & 1
\\ x_3 & 0 & -x_1\\
0 & x_3 &
-x_2\end{bmatrix}}B^3\xrightarrow{\phi_2^*=\begin{bmatrix}-x_3 &
x_2 & -x_1\end{bmatrix}}B\to 0.
\end{equation*}}

$x_2x_3-x_3x_2=-x_2,\ \ x_3x_1-x_1x_3=x_1+x_2,\ \ x_1x_2-x_2x_1=0$
(not $\mathcal{SAS}$) :

{\footnotesize
\begin{equation*}
0\rightarrow B\xrightarrow{\phi_2=\begin{bmatrix}-x_3+2 & x_2 &
-x_1
\end{bmatrix}} B^3 \xrightarrow{\phi_1=\begin{bmatrix}x_2 & -x_1 & 0
\\ x_3-1 & -1 & -x_1\\
0 & x_3-1 & -x_2\end{bmatrix}}B^3
\xrightarrow{\phi_0=\begin{bmatrix}x_1\\
x_2\\
x_3\end{bmatrix}}B\xrightarrow{\epsilon}K\to 0,
\end{equation*}}

{\footnotesize
\begin{equation*}
0\rightarrow B\xrightarrow{\phi_0^*=\begin{bmatrix}x_1
\\ x_2\\
x_3\end{bmatrix}}B^3 \xrightarrow{\phi_1^*=\begin{bmatrix}x_2 &
-x_1 & 0
\\ x_3-1 & -1 & -x_1\\
0 & x_3-1 &
-x_2\end{bmatrix}}B^3\xrightarrow{\phi_2^*=\begin{bmatrix}-x_3+2 &
x_2 & -x_1\end{bmatrix}}B\to 0.
\end{equation*}}

$x_2x_3-x_3x_2=x_3,\ \ x_3x_1-x_1x_3=x_3,\ \ x_1x_2-x_2x_1=0 $
(not $\mathcal{SAS}$):

{\tiny
\begin{equation*}
0\rightarrow B\xrightarrow{\phi_2=\begin{bmatrix}-x_3 & x_2-1 &
-x_1-1
\end{bmatrix}} B^3 \xrightarrow{\phi_1=\begin{bmatrix}x_2 & -x_1 & 0
\\ x_3 & 0 & -x_1-1\\
0 & x_3 & -x_2+1\end{bmatrix}}B^3
\xrightarrow{\phi_0=\begin{bmatrix}x_1\\
x_2\\
x_3\end{bmatrix}}B\xrightarrow{\epsilon}K\to 0,
\end{equation*}}

{\tiny
\begin{equation*}
0\rightarrow B\xrightarrow{\phi_0^*=\begin{bmatrix}x_1
\\ x_2\\
x_3\end{bmatrix}}B^3 \xrightarrow{\phi_1^*=\begin{bmatrix}x_2 &
-x_1 & 0
\\ x_3 & 0 & -x_1-1\\
0 & x_3 &
-x_2+1\end{bmatrix}}B^3\xrightarrow{\phi_2^*=\begin{bmatrix}-x_3 &
x_2-1 & -x_1-1\end{bmatrix}}B\to 0.
\end{equation*}}

\end{center}

\end{example}

Next we present an algebra essentially regular in the sense of
Gaddis (see \cite{Gaddis} and \cite{Gaddis2}) but not
$\mathcal{SAS}$. A $\mathbb{N}$-filtered algebra $A$ is
essentially regular if and only if $Gr(A)$ is $AS$ (see
Proposition 2.3.7 in \cite{Gaddis}).

\begin{example}
Consider the algebra $\mathcal{U}:=K\{ x,y\}/\langle
yx-xy+y\rangle$, by Corollary 2.3.14 in \cite{Gaddis},
$\mathcal{U}$ is essentially regular of global dimension $2$.
Clearly $\mathcal{U}$ is not Artin-Schelter, actually we will show
that $\mathcal{U}$ is not $\mathcal{SAS}$. The semi-graduation of
$\mathcal{U}$ is as in Example \ref{example18.5.3}, so
$\mathcal{U}_{\geq 1}=\oplus_{p\geq 1}\mathcal{U}_p$ and this
ideal coincides with the two-sided ideal of $\mathcal{U}$
generated by $x$ and $y$. Observe that
$_\mathcal{U}(\mathcal{U}/\mathcal{U}_{\geq 1})\cong
_\mathcal{U}K$, where the structure of left $\mathcal{U}$-module
for $K$ is given by the canonical projection
$\epsilon:\mathcal{U}\to K$ (the same is true for the right
structure, $(\mathcal{U}/\mathcal{U}_{\geq 1})_\mathcal{U}\cong
K_\mathcal{U}$). The following sequence is a free resolution of
$_\mathcal{U}K$:

\begin{equation*}
0\rightarrow \mathcal{U}\xrightarrow{\phi_1=\begin{bmatrix}y & 1-x
\end{bmatrix}} \mathcal{U}^2 \xrightarrow{\phi_0=\begin{bmatrix}x
\\ y\end{bmatrix}}\mathcal{U} \xrightarrow{\epsilon}K\to 0.
\end{equation*}
This statement can be proved using \texttt{SPBWE} or simply by
hand. In fact, $\epsilon$ is clearly surjective; $\phi_1$ is
injective since if $\phi_1(u)=0$ for $u\in \mathcal{U}$, then
$u\begin{bmatrix}y & 1-x
\end{bmatrix}=0$, so $u=0$ since
$\mathcal{U}$ is a domain.
$Im(\phi_0)=\ker(\epsilon)=\mathcal{U}_{\geq 1}$ since
\begin{center}
$\begin{bmatrix}u & v\end{bmatrix}\begin{bmatrix}x \\
y\end{bmatrix}=ux+vy$, with $u,v\in \mathcal{U}$.
\end{center}
$Im(\phi_1)\subseteq \ker(\phi_0)$ since $\phi_1\phi_0=0$:
\begin{center}
$\begin{bmatrix}y & 1-x
\end{bmatrix}\begin{bmatrix}x
\\ y\end{bmatrix}=yx+(1-x)y=0$.
\end{center}
Now, $\ker(\phi_0)\subseteq Im(\phi_1)$: In fact, let
$\begin{bmatrix}u & v\end{bmatrix}\in \ker(\phi_0)$, then
$ux+vy=0$; let
\begin{center}
$u=u_0+u_1x+u_2y+u_3x^2+u_4xy+u_5y^2+\cdots$,

$v=v_0+v_1x+v_2y+v_3x^2+v_4xy+v_5y^2+\cdots$,
\end{center}
from $ux+vy=0$ we conclude that all terms of $u$ involving only
$x's$ have coefficient equals zero, i.e., $u=py$ for some $p\in
\mathcal{U}$, hence $vy=-pyx=-p(xy-y)=p(1-x)y$, but since $A$ is a
domain, $v=p(1-x)$, whence,
\begin{center}
$\begin{bmatrix}u & v\end{bmatrix}=p\begin{bmatrix}y &
1-x\end{bmatrix}\in Im(\phi_1)$.
\end{center}
Now we apply $Hom_\mathcal{U}(-,\, _\mathcal{U}\mathcal{U})$ and
we get the complex of right $\mathcal{U}$-modules
\begin{equation*}
0\rightarrow
Hom_\mathcal{U}(K,\mathcal{U})\xrightarrow{\epsilon^*}
Hom_\mathcal{U}(\mathcal{U}, \mathcal{U})
\xrightarrow{\phi_0^*}Hom_{\mathcal{U}} (\mathcal{U}^2,
\mathcal{U})
\xrightarrow{\phi_1^*}Hom_{\mathcal{U}}(\mathcal{U},\mathcal{U})\to
0.
\end{equation*}
As in Example \ref{example18.4.5},
$Hom_\mathcal{U}(K,\mathcal{U})=0$, $Hom_\mathcal{U}(\mathcal{U},
\mathcal{U})\cong \mathcal{U}$ and $Hom_{\mathcal{U}}
(\mathcal{U}^2, \mathcal{U})\cong \mathcal{U}^2$, so we obtain the
complex
\begin{equation*}
0\rightarrow \mathcal{U} \xrightarrow{\phi_0^*=\begin{bmatrix}x
\\ y\end{bmatrix}}\mathcal{U}^2
\xrightarrow{\phi_1^*=\begin{bmatrix}y & 1-x
\end{bmatrix}}\mathcal{U}\to 0,
\end{equation*}
with $\phi_0^*$ is injective. Moreover,
$Im(\phi_0^*)=\ker(\phi_1^*)$: Indeed, it is clear that
$Im(\phi_0^*)\subseteq \ker(\phi_1^*)$; let $\begin{bmatrix}u \\ v
\end{bmatrix}\in \ker(\phi_1^*)$, then $yu+(1-x)v=0$, from this by
a direct computation we get that $v=q'y$ for some $q'\in
\mathcal{U}$, but again by a direct computation it is easy to show
that given a polynomial $q'$ there exists $q\in \mathcal{U}$ such
that $q'y=yq$, whence $v=yq$ for some $q\in \mathcal{U}$. Hence,
$yu+(1-x)yq=0$ implies $y(u-xq)=0$, so $u=xq$. Therefore,
$\ker(\phi_1^*)\subseteq Im(\phi_0^*)$ and we have proved the
claimed equality.

Thus,
$Ext_{\mathcal{U}}^0(K,\mathcal{U})=Hom_\mathcal{U}(K,\mathcal{U})=0$,
$Ext_{\mathcal{U}}^1(K,\mathcal{U})=0$ and
$Ext_{\mathcal{U}}^2(K,\mathcal{U})=\mathcal{U}/Im(\phi_1^*)\ncong
K_{\mathcal{U}}$. In fact, suppose there exists a right
$\mathcal{U}$-module isomorphism
$\mathcal{U}/Im(\phi_1^*)\xrightarrow{\alpha}K_{\mathcal{U}}$, let
$\alpha(\overline{1}):=\lambda$, then $\alpha(\overline{1})\cdot
(1-x)=\lambda\cdot (1-x)$, so $\alpha(\overline{1}\cdot
(1-x))=\lambda$, i.e., $\alpha(\overline{0})=0=\lambda$, whence
$\alpha=0$, a contradiction. We conclude that $\mathcal{U}$ is not
$\mathcal{SAS}$.
\end{example}

Now we will show an algebra that is $\mathcal{SAS}$ but is not
essentially regular.

\begin{example}
Let $\mathcal{S}$ be the algebra of Theorem 4.0.7 in \cite{Gaddis}
defined by
\begin{center}
$B:=K\{x,y\}/\langle yx-1\rangle$.
\end{center}
$B$ has a semi-graduation as in Example \ref{example18.5.3}. It is
known that ${\rm gld}(B)=1$ (see \cite{Gaddis}, Proposition 4.1.1
and also \cite{Bavula3}) and it is clear that $B/B_{\geq1}=0$.
Thus, the condition (iv) in Definition \ref{definition18.5.1}
trivially holds and $B$ is $\mathcal{SAS}$. Now observe that
$Gr(B)\cong R_{yx}$, where $R_{yx}$ is defined in \cite{Gaddis} by
\begin{center}
$R_{yx}:=K\{x,y\}/\langle yx\rangle$.
\end{center}
According to Corollary 2.3.14 in \cite{Gaddis}, $B$ is not
essentially regular.
\end{example}

We conclude the list of examples with an algebra that is not
essentially regular neither $\mathcal{SAS}$.

\begin{example}
Let $B:=R_{yx}$. Observe that $B$ is a finitely graded algebra
with graduation as in in Example \ref{example18.5.3}; by a direct
computation we get the following exact sequences
\begin{equation*}
0\rightarrow B\xrightarrow{\phi_1=\begin{bmatrix}y & 0
\end{bmatrix}} B^2 \xrightarrow{\phi_0=\begin{bmatrix}x
\\ y\end{bmatrix}}B \xrightarrow{\epsilon}K\to 0,
\end{equation*}
\begin{equation*}
0\rightarrow B \xrightarrow{\phi_0^*=\begin{bmatrix}x
\\ y\end{bmatrix}}B^2
\xrightarrow{\phi_1^*=\begin{bmatrix}y & 0
\end{bmatrix}}B\to 0.
\end{equation*}
According to \cite{Gaddis}, ${\rm gld}(B)=2$, but note that
$B/yB\ncong _BK$. In fact, suppose there exists a right $B$-module
isomorphism $B/yB\xrightarrow{\alpha}K_{B}$, let
$\alpha(\overline{1}):=\lambda$, then $\alpha(\overline{1})\cdot
x=\lambda\cdot x$, so $\alpha(\overline{x})=0$, i.e.,
$\overline{x}=\overline{0}$ but clearly $x\notin yB$. This says
that $B$ is not $\mathcal{SAS}$ neither essentially regular.
\end{example}

The previous examples induce the following general result.

\begin{theorem}\label{theorem18.4.12}
Let $K$ be a field and $A:=\sigma(R)\langle x_1,\dots,x_n\rangle$
be a bijective skew $PBW$ extension that satisfies the following
conditions:
\begin{enumerate}
\item[\rm (i)]$R$ and $A$ are $K$-algebras.
\item[\rm (ii)]${\rm lgld}(R)<\infty$.
\item[\rm (iii)]$R$ is a $FSG$ ring with semi-graduation $R=\bigoplus_{p\geq 0}R_p$.
\item[\rm (iv)]$R$ is connected, i.e., $R_0=K$.
\item[\rm (v)]For $1\leq i\leq n$, $\sigma_i,\delta_i$ in Proposition \ref{sigmadefinition} are homogeneous, and there exist $i,j$ such that the parameter
$d_{ij}$ in $(\ref{equation1.2.1})$ satifies $d_{ij}\in K-\{0\}$.
\end{enumerate}
Then, $B$ is a $\mathcal{SAS}$ algebra.
\end{theorem}
\begin{proof}
First note that $A$ is $FSG$ and connected with semi-graduation
\begin{center}
$A_0:=K$, $A_p:=_K\langle R_qx^\alpha\mid q+|\alpha|=p\rangle$ for
$p\geq 1$,
\end{center}
with $x^\alpha$ as in Definition \ref{gpbwextension} and
$|\alpha|:=\alpha_1+\cdots+\alpha_n$. Observe that if
$r_1,\dots,r_m$ generate $R$ as $K$-algebra, then
$r_1,\dots,r_m,x_1,\dots,x_n$ generate $A$ as $K$-algebra.
Moreover, $\dim_K A_p<\infty$ for every $p\geq 0$.

We know that ${\rm lgld}(A)<\infty$ (see \cite{lezamareyes1}). The
condition (v) in the statement of the theorem says that $A/A_{\geq
1}=0$, so the condition (iv) in Definition \ref{definition18.5.1}
trivially holds.
\end{proof}

\begin{remark}
From the general homological properties of skew $PBW$ extensions
we know that if $R$ is a noetherian domain, then $A$ is a
noetherian domain (\cite{lezamareyes1}). This agrees with the
question of Problem 4.
\end{remark}

\section{Zariski cancellation problem}

The Zariski cancellation problem (ZCP) arises in commutative
algebra and can be formulated in the following way: Let $K$ be a
field, $A:=K[t_1,\dots,t_{n}]$ be the commutative algebra of usual
polynomials and $B$ be a commutative $K$-algebra,
\begin{center}
if $A[t]\cong B[t]$, then $A\cong B$?
\end{center}
The ZCP has very interesting connections with some famous
classical problems: The Automorphism problem, the Dixmier
conjecture, the Jacobian conjecture, among some others, see a
discussion in \cite{Essen}. Recently, the ZCP has been considered
for non-commutative algebras, in \cite{BellZhang} Bell and Zhang
studied the Zariski cancellation problem for some non-commutative
Artin-Schelter regular algebras; other works on this problem are
\cite{Zhangetal}, \cite{Zhangetal2}, \cite{LezamaZhang},
\cite{HelbertZhang}. In this section we will formulate a problem
related to the Zariski cancellation problem in the context of
$\mathcal{SAS}$ algebras.

Despite of the most important results about the cancellation
problem are for algebras over fields (see for example Theorem
\ref{theorem20.2.1}), the general recent formulation of this
problem is for $R$-algebras, where $R$ is an arbitrary commutative
domain (Definition \ref{definition20.1.1}). Thus, the general
definition below includes the algebras over fields as well as the
particular case of $\mathbb{Z}$-algebras, i.e., arbitrary rings.

\begin{definition}[\cite{BellZhang}]\label{definition20.1.1}
Let $R$ be a conmutative domain and let $A$ be a $R$-algebra.
\begin{enumerate}
\item[\rm (i)]$A$ is cancellative if for every $R$-algebra $B$,
\begin{center}
$A[t]\cong B[t]\Rightarrow A\cong B$.
\end{center}
\item[\rm (ii)]$A$ is strongly cancellative if for any $d\geq 1$ and every $R$-algebra
$B$,
\begin{center}
$A[t_1,\dots,t_d]\cong B[t_1,\dots,t_d]\Rightarrow A\cong B$.
\end{center}
\item[\rm (iii)]$A$ is universally cancellative if for any $R$-flat finitely generated commutative
domain $S$ such that $S/I\cong R$ for some ideal $I$ of $S$, and
any $R$-algebra $B$
\begin{center}
$A\otimes S\cong B\otimes S\Rightarrow A\cong B$,
\end{center}
where the tensor product is over $R$.
\end{enumerate}
\end{definition}

\begin{remark}\label{remark16.4.2}
(i) All isomorphisms in the previous definition are isomorphisms
of $R$-algebras, and hence, these notions depend on the ring $R$.

(ii) Observe that the commutative Zariski cancellation problem
asks if the $K$-algebra of polynomials $K[t_1,\dots,t_n]$ is
cancellative. Abhyankar-Eakin-Heinzer in \cite{Abhyankar} proved
that $K[t_1]$ is cancellative (actually, they proved that every
commutative finitely generated domain of Gelfand-Kirillov
dimension one is cancellative, see Corollary 3.4 in
\cite{Abhyankar}); Fujita in \cite{Fujita2} and Miyanishi-Sugie in
\cite{Miyanishi} proved that if $char K=0$, then $K[t_1,t_2]$ is
cancellative; if $char K>0$, Russell in \cite{Russell} proved that
$K[t_1,t_2]$ is cancellative. Recently, in 2014, Gupta proved that
if $n\geq 3$ and $char K>0$ then $K[t_1,\dots,t_n]$ is not
cancellative (see \cite{Gupta}, \cite{Gupta2}). The ZCP problem
for $K[t_1,\dots,t_n]$ remains open for $n\geq 3$ and $char K=0$.
\end{remark}

\begin{proposition}[\cite{BellZhang}, Remark 1.2]
For any $R$-algebra $A$,
\begin{center}
universally cancellative $\Rightarrow$ strongly cancellative
$\Rightarrow$ cancellative.
\end{center}
\end{proposition}

For the investigation of the ZCP problem for a given
non-commutative algebra $A$ have been used some subalgebras of
$A$. The first one is the center.

\begin{theorem}[\cite{BellZhang}, Proposition 1.3]\label{theorem20.2.1}
Let $K$ be a field and $A$ be a $K$-algebra. If $Z(A)=K$, then $A$
is universally cancellative, and hence, cancellative.
\end{theorem}

In \cite{LezamaHelbert2} was computed the center of many
$K$-algebras interpreted as skew $PBW$ extensions, some of them
with trivial center, and hence, such algebras are cancellative.

Another subalgebra involved in the study of the ZCP problem is the
Makar-Limanov invariant.

\begin{definition}
Let $A$ be a $R$-algebra.
\begin{enumerate}
\item[\rm (i)]${\rm Der}(A)$ denotes the collection of $R$-derivations of
$A$ and ${\rm LND}(A)$ the collection of locally nilpotent
$R$-derivations of $A$; $\delta\in {\rm Der}(A)$ is locally
nilpotent if given $a\in A$ there exists $n\geq 1$ such that
$\delta^n(a)=0$.
\item[\rm (ii)]The Makar-Limanov invariant of $A$ is defined to be
\[
{\rm ML}(A):=\bigcap_{\delta\in {\rm LND}(A)}\ker(\delta).
\]
\end{enumerate}
\end{definition}

\begin{theorem}[\cite{BellZhang}, Theorems 3.3 and 3.6]\label{theorem16.4.17}
Let $A$ be a $R$-algebra that is a finitely generated domain of
finite Gelfand-Kirillov dimension.
\begin{enumerate}
\item[\rm (i)]If ${\rm ML}(A[t])=A$, then $A$ is cancellative.
\item[\rm (ii)]If $char R=0$ and ${\rm ML}(A)=A$, then $A$ is cancellative.
\end{enumerate}
\end{theorem}

The effective computation of ${\rm ML}(A)$ is in general a
difficult task, therefore other strategies and techniques have
been introduced in order to investigate the cancellation property
for non-commutative algebras. For example, setting algebraic
condition on $A$, it is interesting to know if $A$ becomes
cancellative. In this direction recently have been proved the
following results.

\begin{theorem}[\cite{LezamaZhang}, Theorem 4.1]
Let $A$ be a $R$-algebra. If $A$ is left $($or right$)$ artinian,
then $A$ is strongly cancellative, and hence, cancellative.
\end{theorem}

\begin{theorem}[\cite{HelbertZhang}, Theorem 0.1]
Let $A$ be a $K$-algebra, $K$ a field with $charK=0$. If $A$ is a
noetherian $($left and right$)$ $AS$ algebra generated in degree
$1$ with ${\rm gld}(A)=3$ and $A$ is not $PI$, then $A$ is
cancellative.
\end{theorem}

Observe that if in the previous theorem we could remove the
condition \textit{not $PI$}, then the classical commutative ZCP
problem would be solved for $n=3$. With this, we can formulate our
last problem in the next subsection.

\subsection{Problem 5}

\begin{itemize}
\item Let $B$ be a $\mathcal{SAS}$ algebra over a field
$K$ with $char(K)=0$. Assume that $B$ is noetherian $($left and
right$)$, generated in degree $1$ and ${\rm gld}(B)=3$. Is $B$
cancellative?
\end{itemize}



\end{document}